\documentclass{amsart}
\usepackage{amsfonts}

\usepackage{graphicx}

\newtheorem{theorem}{Theorem}[section]
\theoremstyle{plain} \numberwithin{equation}{section}

\newtheorem{corollary}[theorem]{Corollary}

\newtheorem{definition}[theorem]{Definition}
\newtheorem{example}[theorem]{Example}

\newtheorem{lemma}[theorem]{Lemma}

\newtheorem{proposition}[theorem]{Proposition}
\newtheorem{remark}[theorem]{Remark}

\begin{document}
\title{The geodesic problem in quasimetric spaces}
\author{Qinglan Xia}
\address{University of California at Davis\\
Department of Mathematics\\
Davis,CA,95616} \email{qlxia@math.ucdavis.edu}
\urladdr{http://math.ucdavis.edu/\symbol{126}qlxia}
\thanks{This work is supported by an NSF grant DMS-0710714.}

\begin{abstract}
In this article, we study the geodesic problem in a generalized
metric space, in which the distance function satisfies a relaxed
triangle inequality $d(x,y)\leq \sigma (d(x,z)+d(z,y))$ for some
constant $\sigma \geq 1$, rather than the usual triangle inequality.
Such a space is called a quasimetric space. We show that many
well-known results in metric spaces (e.g. Ascoli-Arzel\`{a} theorem)
still hold in quasimetric spaces. Moreover, we explore conditions
under which a quasimetric will induce an intrinsic metric. As an
example, we introduce a family of quasimetrics on the space of
atomic probability measures. The associated intrinsic metrics
induced by these quasimetrics coincide with the $d_{\alpha}$ metric
studied early in the study of branching structures arisen in
ramified optimal transportation. An optimal transport path between
two atomic probability measures typically has a ``tree shaped"
branching structure. Here, we show that these optimal transport
paths turn out to be geodesics in these intrinsic metric spaces.
\end{abstract}

\maketitle

\section{\protect\bigskip Introduction}

This article aims at studying some classical analysis problems in
semimetric spaces, in which the distance is not required to satisfy
the triangle inequity. During the author's recent study of optimal
transport path between probability measures, he observes that there
exists a family of very interesting semimetrics on the space of
atomic probability measures. These semimetrics satisfy a relaxed
triangle inequality $d\left( x,y\right) \leq \sigma \left( d\left(
x,z\right) +d\left( z,y\right) \right) $ for some constant $\sigma
\geq 1$, rather than the usual triangle inequality. Such semimetric
spaces are called quasimetric spaces \footnote{When this article was
submitted, the author used the term ``nearmetric" as in \cite{fks}
instead of ``quasimetric". Later, Professor Nigel Kalton kindly let
the author know the term ``quasimetric" used in the book
\cite{Hein}. Thus, in the final version of the article,  we replaced
the previous term ``nearmetric" with this more suitable term
``quasimetric".} in \cite{Hein}. Moreover, these family of
quasimetrics indeed induce a family of intrinsic metrics on the
space of atomic probability measures. Furthermore, optimal transport
paths studied early in
\cite{xia1},\cite{xia2},\cite{xia3},\cite{xia4},\cite{xia5} etc turn
out to be exactly geodesics in these induced metric spaces. This
observation motivates us to study the geodesic problem in
quasimetric spaces in this article. Other closely related works on
ramified optimal transportation may be found in
\cite{BCM},\cite{DH},\cite{gilbert},\cite{msm} etc.

This article is organized as follows. In section 2, we first
introduce the concept as well as some basic properties of
quasimetric spaces, then we extend some well-known results (e.g.
Ascoli-Arzel\`{a} theorem) about continuous functions in metric
spaces to continuous functions in quasimetric spaces. After that, in
section 3, we consider the geodesic problem in quasimetric spaces.
We show that every continuous quasimetric will induce an intrinsic
pseudometric on the space. In case that the quasimetric is nice
enough (e.g. either ``ideal'' or ``perfect'' in the sense of Definition \ref%
{def_ideal} or Definition \ref{perfect}), then the quasimetric will
indeed induce an intrinsic metric. In the end, we spend the last
section in discussing our motivation example: optimal transport
paths between atomic probability measures. We first introduce a
family of quasimetrics on the space of atomic probability measures.
Each of these quasimetric is both ideal and perfect, and thus it
induces an intrinsic metric on the space of atomic probability
measures. We showed that the $d_{\alpha }$-metrics introduced in\
\cite{xia1} is simply the intrinsic metrics induced by these
quasimetrics. Furthermore, each geodesic in these length spaces
corresponds to an optimal transport path studied in \cite{xia1}.

\section{Continuous maps in quasimetric spaces}

\subsection{Quasimetric Spaces}

\begin{definition}
\label{near_metric_def}Let $X$ be any nonempty set. A function
$J:X\times X\rightarrow \mathbb{R}$ is called a quasimetric if for
any $x,y,z\in X$, we have

\begin{enumerate}
\item \label{condition_1}(non-negativity) $J\left( x,y\right) \geq 0$;

\item \label{condition_2}(identity of indiscernibles) $J\left( x,y\right) =0$
if and only if $x=y$

\item \label{condition_3}(symmetry) $J\left( x,y\right) =J\left( y,x\right)
; $

\item \label{condition_4}(relaxed triangle inequality) $J\left( x,y\right)
\leq \sigma \left[ J\left( x,z\right) +J\left( z,y\right) \right] $ for some
constant $\sigma \geq 1$.
\end{enumerate}

When $J$ is a quasimetric on $X$, the pair $\left( X,J\right) $ is
called \textit{a quasimetric space}. Let $\sigma \left( J\right) $
denote the smallest number $\sigma $ satisfying condition
(\ref{condition_4}).
\end{definition}

Every metric space is clearly a quasimetric space with $\sigma =1$.

\begin{example}
Suppose $d$ is a metric on a nonempty set $X$. Then, for any $\beta
>1,\lambda \geq 0,\mu >0$, $J\left( x,y\right) =\lambda d(x,y)+\mu d\left(
x,y\right) ^{\beta }$ is typically not a metric on $X$. However, $J$
defines a quasimetric on $X$ with $\sigma \left( J\right) \leq
2^{\beta -1}$. Indeed,\
\begin{eqnarray*}
J\left( x,y\right) &=&\lambda d(x,y)+\mu d\left( x,y\right) ^{\beta } \\
&\leq &\lambda \left[ d(x,z)+d(y,z)\right] +\mu \left[ d(x,z)+d(y,z)\right]
^{\beta } \\
&\leq &\lambda \left[ d(x,z)+d(y,z)\right] +2^{\beta -1}\mu \left[
d(x,z)^{\beta }+d(y,z)^{\beta }\right] \\
&\leq &2^{\beta -1}\left[ J\left( x,z\right) +J\left( z,y\right) \right] .
\end{eqnarray*}
\end{example}

In section 4, we will provide a family of interesting quasimetrics
on the space of atomic probability measures.

More generally, suppose $J$ is a distance function on $X$ satisfying
conditions (\ref{condition_1}),(\ref{condition_2}),(\ref{condition_3}) in
Definition \ref{near_metric_def}. For each $n$, let $\sigma _{n}\left(
J\right) $ be the smallest number $\sigma _{n}\geq 1$ satisfying
\begin{equation}
J\left( x_{1},x_{n+1}\right) \leq \sigma _{n}\sum_{i=1}^{n}J\left(
x_{i},x_{i+1}\right) ,  \label{order_n}
\end{equation}%
for any $x_{1}$,$\cdots ,x_{n+1}\in X$. In particular, $\sigma _{1}\left(
J\right) =1$ and $\sigma _{2}\left( J\right) =\sigma \left( J\right) $.

\begin{lemma}
\label{sigma_n}Suppose $\left( X,J\right) $ is a quasimetric space.
Then, for each $n$,
\begin{equation*}
\sigma _{n}\left( J\right) \leq \sigma \left( J\right) ^{n-1}.
\end{equation*}
\end{lemma}

\begin{proof}
We show this using the mathematical induction. It is trivial when $n=1$ or $%
2 $. Then, from condition (\ref{condition_4}), we see that for any $n $ and
any points $\left\{ x_{1},x_{2},\cdots ,x_{n}\right\} $ in $X$, we have
\begin{eqnarray*}
J\left( x_{1},x_{n}\right) &\leq &\sigma \left( J\right) \left( J\left(
x_{1},x_{n-1}\right) +J\left( x_{n-1},x_{n}\right) \right) \\
&\leq &\sigma \left( J\right) \left( \sigma \left( J\right)
^{n-2}\sum_{i=1}^{n-2}J\left( x_{i},x_{i+1}\right) +J\left(
x_{n-1},x_{n}\right) \right) \\
&\leq &\sigma \left( J\right) ^{n-1}\sum_{i=1}^{n-1}J\left(
x_{i},x_{i+1}\right) \text{ since }\sigma \left( J\right) \geq 1.
\end{eqnarray*}%
Therefore, $\sigma _{n}\left( J\right) \leq \sigma \left( J\right) ^{n-1}$
for all $n$.
\end{proof}

\begin{proposition}
\bigskip Suppose $\left( X,J\right) $ is a quasimetric space. Then, for each $%
n$ and $m$ in $\mathbb{N}$,
\begin{equation*}
\sigma _{nm}\left( J\right) \leq \sigma _{n}\left( J\right) \sigma
_{m}\left( J\right) .
\end{equation*}
\end{proposition}

\begin{proof}
Note that, for any $\left\{ x_{1},x_{2},\cdots ,x_{mn+1}\right\} $ in $X$,
from (\ref{order_n}), we have
\begin{eqnarray*}
& &J\left( x_{1},x_{mn+1}\right) \\
&\leq &\sigma _{n}\left( J\right) \left( J\left( x_{1},x_{m+1}\right)
+J\left( x_{m+1},x_{2m+1}\right) +\cdots +J\left(
x_{(n-1)m+1},x_{nm+1}\right) \right) \\
&\leq &\sigma _{n}\left( J\right) \left( \sigma _{m}\left( J\right)
\sum_{i=1}^{m}J\left( x_{i},x_{i+1}\right) +\cdots +\sigma _{m}\left(
J\right) \sum_{i=(n-1)m+1}^{nm}J\left( x_{i},x_{i+1}\right) \right) \\
&=&\sigma _{n}\left( J\right) \sigma _{m}\left( J\right)
\sum_{i=1}^{nm}J\left( x_{i},x_{i+1}\right) .
\end{eqnarray*}%
Therefore,
\begin{equation*}
\sigma _{nm}\left( J\right) \leq \sigma _{n}\left( J\right) \sigma
_{m}\left( J\right) .
\end{equation*}
\end{proof}

Clearly, $\sigma _{n}\left( J\right) $ is nondecreasing as $n$ increases.
Thus, we define
\begin{equation}
\sigma _{\infty }\left( J\right) :=\lim_{n}\sigma _{n}\left( J\right)
\label{ideal}
\end{equation}%
for any quasimetric $J$ on $X$.

\begin{definition}
\label{def_ideal}Suppose $J$ is a quasimetric on $X$. If $\sigma
_{\infty }\left( J\right) <\infty $, then \thinspace $J$ is called
an ideal quasimetric on $X$.
\end{definition}

Note that $J$ is an ideal quasimetric if and only if for some $\sigma \geq 1$%
,
\begin{equation}
J\left( x,y\right) \leq \sigma \sum_{i=1}^{n}J\left(
x_{i},x_{i+1}\right), \label{ideal_quasimetric}
\end{equation}%
for any finitely many points $x_{1}$,$\cdots ,x_{n+1}\in X$ with $x_{1}=x$, $%
x_{n+1}=y $. The smallest $\sigma $ satisfying
(\ref{ideal_quasimetric}) is just $\sigma _{\infty }\left( J\right)
$.

\bigskip A sequence $\left\{ x_{n}\right\} $ is \textit{convergent} to $x$
in a quasimetric space $\left( X,J\right) $ if $J\left(
x_{n},x\right) \rightarrow 0$, and we denote it by
$x_{n}\overset{J}{\rightarrow }x$. A sequence $\left\{ x_{n}\right\}
$ is \textit{Cauchy in }$\left( X,J\right) $
if for any $\epsilon >0$, there exists an $\,N\in \mathbb{N}$ such that $%
J\left( x_{n},x_{m}\right) \leq \epsilon $ for all $n,m\geq N$. \ Since $%
J\left( x_{n},x_{m}\right) \leq \sigma \left( J\right) \left(
J\left( x_{n},x\right) +J\left( x,x_{m}\right) \right) $, it follows
that every convergent sequence in $\left( X,J\right) $ is a Cauchy
sequence. If every Cauchy sequence in $\left( X,J\right) $ is
convergent, then we say $J$ is a \textit{complete} quasimetric on
$X$. A quasimetric $J$ on $X$ always gives a
topology on $X$ where a subset $A$ is closed if it contains every point $%
a\in X$ for which there is some sequence $a_{i}\in A$ with $%
\lim_{i\rightarrow \infty }J\left( a_{i},a\right) =0$.

\begin{definition}
A quasimetric $J$ on $X$ is \textit{continuous} if for any
convergent
sequences $x_{n}\overset{J}{\rightarrow }x$, $y_{n}\overset{J}{\rightarrow }%
y $, we have
\begin{equation}
J\left( x_{n},y_{n}\right) \rightarrow J\left( x,y\right) ,\text{ as }%
n\rightarrow \infty \text{.}  \label{J_continuity}
\end{equation}%
If for any convergent sequences $x_{n}\overset{J}{\rightarrow }x$, $y_{n}%
\overset{J}{\rightarrow }y$, we have
\begin{equation}
J\left( x,y\right) \leq \liminf_{n}J\left( x_{n},y_{n}\right) ,
\label{lower_semicontinuity}
\end{equation}%
then we say $J$ is lower semicontinuous.
\end{definition}

For instance, suppose $J$ satisfies conditions (\ref{condition_1}),(\ref%
{condition_2}),(\ref{condition_3}) in Definition \ref{near_metric_def}, and
also the following condition%
\begin{equation}
\left| J\left( x,y\right) -J\left( z,w\right) \right| \leq \sigma \left(
J\left( x,z\right) +J\left( w,y\right) \right)  \label{continuous_extra}
\end{equation}%
for any $x,y,z,w\in X$ and some $\sigma \geq 1$. By setting $z=w$, we get $%
J\left( x,y\right) \leq \sigma \left[ J\left( x,z\right) +J\left( z,y\right) %
\right] $, and hence $J$ is a quasimetric on $X$. Also, since for
each $n$,
\begin{equation*}
\left| J\left( x_{n},y_{n}\right) -J\left( x,y\right) \right| \leq \sigma
\left( J\left( x,x_{n}\right) +J\left( y,y_{n}\right) \right) ,
\end{equation*}%
$J$ is automatically satisfying the continuous condition (\ref{J_continuity}%
) in this case. When $J$ is indeed a metric on $X$, then (\ref%
{continuous_extra}) trivially holds.

\subsection{Continuous maps in quasimetric spaces}

In this section, we extend some well-known results (see for instance in \cite%
{hunter} or \cite{Ambrosio1}) about continuous maps in metric spaces
to continuous maps in quasimetric spaces.

Suppose $\left( X,J\right) $ is a quasimetric space, and $K$ is a
compact metric space with a metric $d_{K}$. A map $f:K\rightarrow
\left( X,J\right) $ is \textit{continuous} if $J\left( f\left(
x_{n}\right) ,f\left( x\right) \right) \rightarrow 0$ in $X$
whenever $d_{K}\left( x_{n},x\right) \rightarrow 0$ in $K$ as
$n\rightarrow \infty $. A map $f:K\rightarrow \left( X,J\right) $ is
\textit{uniformly continuous} if for every $\epsilon
>0$, there exists a $\delta >0$ such that $J\left( f\left( x\right) ,f\left(
y\right) \right) \leq \epsilon $ whenever $x,y\in K$ with $d_{K}\left(
x,y\right) \leq \delta $. A map $f:K\rightarrow \left( X,J\right) $ is
\textit{Lipschitz} if there exists a constant $C\geq 0$ such that
\begin{equation*}
J\left( f\left( x\right) ,f\left( y\right) \right) \leq Cd_K(x,y)
\end{equation*}%
for any $x,y\in K$. Let $C\left( K,\left( X,J\right) \right) $ be the family
of all continuous maps from $K$ to $\left( X,J\right) $, and $Lip\left(
K,\left( X,J\right) \right) $ be the family of all Lipschitz maps from $K$
to $\left( X,J\right) $.

\begin{proposition}
\label{uniform_continuity}Suppose $J$ is a continuous quasimetric on
$X$. Then, every continuous map $f:K\rightarrow \left( X,J\right) $
is uniformly continuous.
\end{proposition}

\begin{proof}
Suppose $f:K\rightarrow \left( X,J\right) $ is continuous. If $f$ is not
uniformly continuous, then there exists an $\epsilon >0$, and two sequences $%
\left\{ x_{n}\right\} ,\left\{ y_{n}\right\} $ in $K$ such that $d\left(
x_{n},y_{n}\right) \leq \frac{1}{n}$, but $J\left( f\left( x_{n}\right)
,f\left( y_{n}\right) \right) \geq \epsilon $. By the compactness of $K$ and
taking subsequence if necessary, we may assume that both $\left\{
x_{n}\right\} $ and $\left\{ y_{n}\right\} $ converge to the same point $%
x^{\ast }\in K$. So, by the continuity of $J$ in (\ref{J_continuity}) and
the continuity of $f$ at $x^{\ast }$, we have
\begin{equation*}
0=J\left( f\left( x^{\ast }\right) ,f\left( x^{\ast }\right) \right)
=\lim_{n\rightarrow \infty }J\left( f\left( x_{n}\right) ,f\left(
y_{n}\right) \right) \geq \epsilon .
\end{equation*}%
A contradiction. Thus,\thinspace $f$ must be uniformly continuous.
\end{proof}

For any maps $f,h:K\rightarrow \left( X,J\right) $, let
\begin{equation}
J_{\infty }\left( f,h\right) :=\sup_{x\in K}J\left( f\left( x\right)
,h\left( x\right) \right) .  \label{J_infty}
\end{equation}%
If $J_{\infty }\left( f_{n},f\right) \rightarrow 0$, then we say that $f_{n}$
is \textit{uniformly convergent} to $f$.

\begin{proposition}
\label{J_infty_quasimetric}Suppose $J$ is a quasimetric on $X$. Then, $%
J_{\infty }$ is \ a quasimetric on $C\left( K,\left( X,J\right)
\right) $.
\end{proposition}

\begin{proof}
For any $f,h\in C\left( K,\left( X,J\right) \right) $, by definition (\ref%
{J_infty}), we have $J_{\infty }\left( f,h\right) \geq 0$ and $J_{\infty
}\left( f,h\right) =J_{\infty }\left( h,f\right) $. Also, $J_{\infty }\left(
f,h\right) =0$ if and only if $f\left( x\right) =h\left( x\right) $ for all $%
x\in K$. Moreover, for any $g\in C\left( K,\left( X,J\right) \right) $,
\begin{eqnarray*}
J_{\infty }\left( f,h\right) &=&\sup_{x\in K}J\left( f\left( x\right)
,h\left( x\right) \right) \\
&\leq &\sup_{x\in K}\sigma \left( J\right) \left[ J\left( f\left( x\right)
,g\left( x\right) \right) +J\left( g\left( x\right) ,h\left( x\right)
\right) \right] \\
&\leq &\sigma \left( J\right) \left[ \sup_{x\in K}J\left( f\left( x\right)
,g\left( x\right) \right) +\sup_{x\in K}J\left( g\left( x\right) ,h\left(
x\right) \right) \right] \\
&=&\sigma \left( J\right) \left[ J_{\infty }\left( f,g\right) +J_{\infty
}\left( g,h\right) \right] .
\end{eqnarray*}

Therefore, $\left( C\left( K,\left( X,J\right) \right) ,J_{\infty
}\right) $ is also a quasimetric space.
\end{proof}

\begin{proposition}
\label{limit_of_uniform_continuous}Suppose $\left\{ f_{n}:K\rightarrow
\left( X,J\right) \right\} $ is a sequence of continuous maps. If $J_{\infty
}\left( f_{n},f\right) \rightarrow 0$, then $f$ is also continuous.
\end{proposition}

\begin{proof}
Since $J_{\infty }\left( f_{n},f\right) \rightarrow 0$, for any $\epsilon >0$%
, there exists an $n$ such that
\begin{equation}
\sup_{x\in K}J\left( f_{n}\left( x\right) ,f\left( x\right) \right) \leq
\epsilon /3\text{ }  \label{J_EPSILON_3}
\end{equation}%
For any $x\in K$, since $f_{n}$ is continuous at $x$, there exists a $\delta
=\delta \left( x\right) >0$ such that $J\left( f_{n}\left( x\right)
,f_{n}\left( y\right) \right) \leq \epsilon /3$ whenever $y\in K$ with $%
d_{K}\left( x,y\right) $ $\leq \delta $. Therefore, by lemma \ref{sigma_n}
and (\ref{J_EPSILON_3}), we have
\begin{eqnarray*}
J\left( f\left( x\right) ,f\left( y\right) \right) &\leq &\sigma \left(
J\right) ^{2}\left[ J\left( f\left( x\right) ,f_{n}\left( x\right) \right)
+J\left( f_{n}\left( x\right) ,f_{n}\left( y\right) \right) +J\left(
f_{n}\left( y\right) ,f\left( y\right) \right) \right] \\
&\leq &\epsilon \sigma \left( J\right) ^{2}
\end{eqnarray*}%
and thus $f$ is continuous at every $x\in K$.
\end{proof}

\begin{theorem}
\label{complete_quasimetric}\bigskip Suppose $\left( X,J\right) $ is
a
complete quasimetric space and $J$ is lower semicontinuous. Then, the space $%
\left( C\left( K,\left( X,J\right) \right) ,J_{\infty }\right) $ is
also a complete quasimetric space.
\end{theorem}

\begin{proof}
Let $\left\{ f_{n}\right\} $ be any Cauchy sequence in $C\left( K,\left(
X,J\right) \right) $ with respect to $J_{\infty }$. That is, for any $%
\epsilon >0$, there exists an $N$ such that whenever $m,n\geq N$, we have $%
J_{\infty }\left( f_{n},f_{m}\right) \leq \epsilon $. So, for each $x\in K$,
$\left\{ f_{n}\left( x\right) \right\} $ is Cauchy in $X$. Since $X$ is
complete, $\left\{ f_{n}\left( x\right) \right\} $ converges to some $%
f\left( x\right) \in X$ with respect to $J$. Now,
\begin{eqnarray*}
J_{\infty }\left( f_{n},f\right) &=&\sup_{x\in K}J\left( f_{n}\left(
x\right) ,f\left( x\right) \right) \\
&\leq &\sup_{x\in K}\lim_{m\rightarrow \infty }J\left( f_{n}\left( x\right)
,f_{m}\left( x\right) \right) \text{, because }J\text{ is lower
semicontinuous} \\
&\leq &\limsup_{m\rightarrow \infty }\left[ \sup_{x\in K}J\left( f_{n}\left(
x\right) ,f_{m}\left( x\right) \right) \right] \leq \epsilon
\end{eqnarray*}%
So, $J_{\infty }\left( f_{n},f\right) \rightarrow 0$. By proposition \ref%
{limit_of_uniform_continuous}, $f$ is continuous. Hence, by proposition \ref%
{J_infty_quasimetric}, $J_{\infty }$ is a complete quasimetric on
$C\left( K,\left( X,J\right) \right) $.
\end{proof}

\begin{definition}
A subset $\mathcal{F}$ of $C\left( K,\left( X,J\right) \right)$ is
equicontinuous if for every $x\in K$ and $\epsilon >0$, there is a $\delta
=\delta \left( x,\epsilon \right) >0$, such that whenever $y\in K$ with $%
d_{K}\left( x,y\right) \leq \delta $, we have $J\left( f\left( x\right)
,f\left( y\right) \right) \leq \epsilon $ for all $f\in \mathcal{F}$.
\end{definition}

Now, we have the following Ascoli-Arzel\`{a} theorem in quasimetric
spaces:

\begin{theorem}
Suppose $\left( X,J\right) $ is a complete quasimetric space and $J$
is lower semicontinuous. A subset $\mathcal{F}$ of $\left( C\left(
K,\left( X,J\right) \right) ,J_{\infty }\right) $ is precompact if
and only if it is bounded and equicontinuous.
\end{theorem}

\begin{proof}
Suppose $\mathcal{F}$ is a precompact (i.e. every sequence has a convergent
subsequence) subset of $C\left( K,\left( X,J\right) \right) $. Then, for
each fixed $\epsilon >0$ , there exists a finite subset $\left\{
f_{1},\cdots ,f_{k}\right\} $ of $\mathcal{F}$ such that
\begin{equation}
\mathcal{F}\subset \bigcup\limits_{i=1}^{k}B_{\epsilon /3}\left(
f_{i}\right) ,  \label{finite_cover}
\end{equation}%
where the notation $B_{\epsilon }\left( g\right) =\left\{ h\in C\left(
K,\left( X,J\right) \right) |J_{\infty }\left( g,h\right) <\epsilon \right\}
$. Otherwise, for any finite subset $\left\{ f_{1},\cdots ,f_{k}\right\} $,
there exists an $f_{k+1}\notin \bigcup\limits_{i=1}^{k}B_{\epsilon /3}\left(
f_{i}\right) $, and thus we get a sequence $\left\{ f_{k}\right\} $ in $%
\mathcal{F}$. Since $J_{\infty }\left( f_{m},f_{n}\right) \geq \epsilon /3$
for any $m\neq n$, we know $\left\{ f_{n}\right\} $ does not contain any
Cauchy subsequence, which contradicts to $\mathcal{F}$ being precompact.
Therefore, (\ref{finite_cover}) must be true, which also implies that $%
\mathcal{F}$ is bounded.

Now, for any $x\in K$ and each $f_{i}$ in (\ref{finite_cover}), there exists
a $\delta _{i}>0$ such that whenever $y\in K$ with $d_{K}\left( x,y\right)
<\delta _{i}$, we have $J\left( f_{i}\left( x\right) ,f_{i}\left( y\right)
\right) \leq \frac{\epsilon }{3}$. \ For every $f\in \mathcal{F}$, by (\ref%
{finite_cover}), there is an $1\leq i\leq k$ such that $J_{\infty }\left(
f,f_{i}\right) \leq \frac{\epsilon }{3}$. We conclude that for any $y\in K$
with $d_{K}\left( x,y\right) <\delta =\min \left\{ \delta _{1},\cdots
,\delta _{k}\right\} $, we have
\begin{eqnarray*}
J\left( f\left( x\right) ,f\left( y\right) \right) &\leq &\sigma \left(
J\right) ^{2}\left[ J\left( f\left( x\right) ,f_{i}\left( x\right) \right)
+J\left( f_{i}\left( x\right) ,f_{i}\left( y\right) \right) +J\left(
f_{i}\left( y\right) ,f\left( y\right) \right) \right] \\
&\leq &\epsilon \sigma \left( J\right) ^{2}\text{. }
\end{eqnarray*}%
Therefore, $\mathcal{F}$ is equicontinuous at every $x\in K$.

On the other hand, suppose $\mathcal{F}$ is equicontinuous and bounded.
Then, for any sequence $\left\{ f_{n}\right\} $ in $\mathcal{F}$, by using
the diagonal process and taking subsequence if necessary, we may assume $%
\left\{ f_{n}\right\} $ is convergent to $f$ on a countable dense subset $S$
in $K$. We now prove that $\left\{ f_{n}\right\} $ is Cauchy\ in $C\left(
K,\left( X,J\right) \right) \,\ $with respect to $J_{\infty }$. Indeed, for
any $\epsilon >0$, since $\mathcal{F}$ is equicontinuous and $K$ is compact,
there exists a finite many points $\left\{ r_{1},\cdots ,r_{k}\right\} $ in $%
S$ such that for any $x\in K$, there is a $r_{i}$, such that
\begin{equation*}
J\left( f_{n}\left( x\right) ,f_{n}\left( r_{i}\right) \right) \leq \frac{%
\epsilon }{3}
\end{equation*}%
for all $n$. Now, whenever $m,n$ are large enough, for all $x\in K$,
\begin{eqnarray*}
&&J\left( f_{n}\left( x\right) ,f_{m}\left( x\right) \right) \\
&\leq &\sigma \left( J\right) ^{2}\left[ J\left( f_{n}\left( x\right)
,f_{n}\left( r_{i}\right) \right) +J\left( f_{n}\left( r_{i}\right)
,f_{m}\left( r_{i}\right) \right) +J\left( f_{m}\left( r_{i}\right)
,f_{m}\left( x\right) \right) \right] \\
&\leq &\sigma \left( J\right) ^{2}\epsilon .
\end{eqnarray*}%
Therefore, $\left\{ f_{n}\right\} $ is a Cauchy sequence in $C\left(
K,\left( X,J\right) \right) $. By the completeness of $C\left(
K,\left( X,J\right) \right) $ stated in theorem
\ref{complete_quasimetric}, the
sequence $\left\{ f_{n}\right\} $ is convergent with respect to $J_{\infty }$%
. Thus, $\mathcal{F}$ is precompact.
\end{proof}

\begin{corollary}
\label{Ascoli}Suppose $\left( X,J\right) $ is a complete quasimetric
space and $J$ is lower semicontinuous. A subset $\mathcal{F}$ of
$C\left( K,\left( X,J\right) \right) $ is sequentially compact with
respect to $J_{\infty }$ if and only if it is closed, bounded and
equicontinuous.
\end{corollary}

\section{Intrinsic Metrics induced by quasimetrics}

This section is devoted to study the geodesic problem in a quasimetric space $%
\left( X,J\right) $. Let $\left[ a,b\right] $ be a bounded closed interval.

\begin{definition}
Let $N$ be a natural number. A curve $f\in C\left( \left[ a,b\right] ,\left(
X,J\right) \right) $ is called an $N$-piecewise metric Lipschitz curve in $%
\left( X,J\right) $ if there exists a partition
\begin{equation*}
P_{f}=\left\{ a=a_{0}<a_{1}<\cdots <a_{N}=b\right\}
\end{equation*}%
of $\left[ a,b\right] $ such that for each $i=0,1,\cdots ,N-1$,

\begin{enumerate}
\item $J$ is a metric on the subset $f\left( \left[ a_{i},a_{i+1}\right]
\right) $ of $X$ and

\item the restriction of $f$ on $\left[ a_{i},a_{i+1}\right] $ is Lipschitz.
\end{enumerate}
\end{definition}

Here, requiring $J$ to be a metric on $f\left( \left[ a_{i},a_{i+1}\right]
\right) $ is the same as asking it to satisfy the triangle inequality: $%
J(f(t_{1}),f(t_{2}))\leq J(f(t_{1}),f(t_{2}))+J(f(t_{2}),f(t_{3}))$ for any $%
t_{1},t_{2},t_{3}\in \lbrack a_{i},a_{i+1}]$. Let
\begin{equation*}
\mathcal{P}_{N}\left( \left[ a,b\right] \text{,}\left( X,J\right) \right)
\end{equation*}%
be the family of all $N-$piecewise metric Lipschitz curves in $\left(
X,J\right) $, and $\mathcal{P}\left( \left[ a,b\right] \text{,}\left(
X,J\right) \right) $ be the union of $\mathcal{P}_{N}\left( \left[ a,b\right]
\text{,}\left( X,J\right) \right) $ over all $N$'s.

\subsection{Length of rectifiable curves}

\bigskip Recall that when $\left( X,d\right) $ is a metric space, and $f:%
\left[ a,b\right] \rightarrow \left( X,d\right) $ is a (continuous) curve.
Then, one may define its length as
\begin{equation*}
L\left( f\right) =\sup_{P}V_{P}\left( f\right) \in \left[ 0,+\infty \right] ,
\end{equation*}%
where the supremum is over all partitions $P$ of $\left[ a,b\right] $, and $%
V_{P}\left( f\right) $ is the variation of $f$ over the partition $P=\left\{
a=t_{0}<t_{1}<\cdots <t_{N}=b\right\} $ given by
\begin{equation*}
V_{P}\left( f\right) =\sum_{i=1}^{N}d\left( f\left( t_{i-1}\right) ,f\left(
t_{i}\right) \right) \text{.}
\end{equation*}%
In case $f$ is Lipschitz, an equivalent formula for the length of $f$ is

\begin{equation*}
L\left( f\right) =\int_{a}^{b}\left| \dot{f}\left( t\right) \right| _{d}dt,
\end{equation*}%
where $\left| \dot{f}\left( t\right) \right| _{d}$ is the metric derivative
of $f$ at $f\left( t\right) $ defined by
\begin{equation*}
\left| \dot{f}\left( t\right) \right| _{d}:=\lim_{s\rightarrow t}\frac{%
d\left( f\left( s\right) ,f\left( t\right) \right) }{\left| s-t\right| },
\end{equation*}%
provided the limit exists. When $f$ is Lipschitz, $\left| \dot{f}\left(
t\right) \right| _{d}$ exists almost everywhere, and is bounded and
measurable in $t$.

Now, suppose $\left( X,J\right) $ is a quasimetric space, and $f\in \mathcal{P%
}_{N}\left( \left[ a,b\right] \text{,}\left( X,J\right) \right) $. Then on
each interval $\left[ a_{i},a_{i+1}\right] $, $f:\left[ a_{i},a_{i+1}\right]
\rightarrow \left( X,J\right) $ is a Lipschitz curve in the metric space $%
\left( f\left( \left[ a_{i},a_{i+1}\right] \right) ,J\right) $, and thus the
length of the restriction of $f$ on $\left[ a_{i},a_{i+1}\right] $ is well
defined. As a result, we may define the length of $f$ to be
\begin{equation*}
L\left( f\right) :=\sum_{i=0}^{N-1}L\left( f\lfloor _{\left[ a_{i},a_{i+1}%
\right] }\right) .
\end{equation*}%
In other words, we have

\begin{definition}
For any $f\in \mathcal{P}_{N}\left( \left[ a,b\right] \text{,}\left(
X,J\right) \right) $, the length of $f$ is defined as
\begin{equation*}
L_{J}\left( f\right) :=\int_{a}^{b}\left| \dot{f}\left( t\right) \right|
_{J}dt,
\end{equation*}%
where the metric derivative
\begin{equation*}
\left| \dot{f}\left( t\right) \right| _{J}:=\lim_{s\rightarrow t}\frac{%
J\left( f\left( s\right) ,f\left( t\right) \right) }{\left| s-t\right| }
\end{equation*}%
provided the limit exists. We may simply write $L_{J}\left( f\right) $ as $%
L\left( f\right) $ if $J$ is obvious.
\end{definition}

\begin{lemma}
\label{sequentially_compact_family}Suppose $J$ is a continuous
quasimetric on $X$, $C>0$ is a constant, and $P=\left\{
a=a_{0}<a_{1}<\cdots <a_{N}=b\right\} $ is a partition of the
interval $\left[ a,b\right] $. Then, for any $x,y \in X$, the family
\begin{equation*}
\mathcal{F}=\left\{
\begin{array}{c}
f\in C\left( \left[ a,b\right] ,\left( X,J\right) \right) :\text{ }f\left(
a\right) =x,f\left( b\right) =y, \text{ and } J \text{ is a metric on } \\
f\left( \left[ a_{i},a_{i+1}\right] \right) \text{ and }Lip\left( f\lfloor _{%
\left[ a_{i},a_{i+1}\right] }\right) \leq C, \text{ for each } i=0,\cdots
,N-1%
\end{array}%
\right\}
\end{equation*}%
is a bounded, closed and equicontinuous subset of $C\left( \left[ a,b\right]
,\left( X,J\right) \right) $. Moreover, if $f_{n}$ is uniformly convergent
to $f$ in $J_{\infty }$, then,
\begin{equation*}
L\left( f\right) \leq \liminf_{n}L\left( f_{n}\right) .
\end{equation*}
\end{lemma}

\begin{proof}
For any $g\in \mathcal{F}$ and any $t\in \left[ a,b\right] $, we have $t\in %
\left[ a_{j},a_{j+1}\right] $ for some $j\leq N-1$ and
\begin{eqnarray*}
J\left( g\left( t\right) ,x\right) &=&J\left( g\left( t\right) ,g\left(
a\right) \right) \\
&=&\sigma \left( J\right) ^{j}\left( \sum_{i=0}^{j-1}J\left( g\left(
a_{i}\right) ,g\left( a_{i+1}\right) \right) +J\left( g\left( a_{j}\right)
,g\left( t\right) \right) \right) \\
&\leq &\sigma \left( J\right) ^{j}C\left| t-a\right| \leq C\sigma \left(
J\right) ^{N-1}\left| b-a\right|
\end{eqnarray*}%
Therefore, $\mathcal{F}$ is bounded.

Suppose $\left\{ f_{n}\right\} $ is any convergent sequence in $\mathcal{F}$
with respect to $J_{\infty }$ with $f\in C\left( \left[ a,b\right] ,\left(
X,J\right) \right) $ being the limit. Then, for each fixed $i$, and any $%
t_{1},t_{2},t_{3}\in \left[ a_{i},a_{i+1}\right] $, we have
\begin{equation*}
J\left( f_{n}\left( t_{1}\right) ,f_{n}\left( t_{2}\right) \right) \leq
J\left( f_{n}\left( t_{1}\right) ,f_{n}\left( t_{3}\right) \right) +J\left(
f_{n}\left( t_{3}\right) ,f_{n}\left( t_{2}\right) \right)
\end{equation*}%
and
\begin{equation*}
J\left( f_{n}\left( t_{1}\right) ,f_{n}\left( t_{2}\right) \right) \leq
C\left| t_{1}-t_{2}\right| .
\end{equation*}%
Let $n\rightarrow \infty $, we have $J$ is a metric on $f\left( \left[
a_{i},a_{i+1}\right] \right) $ and $Lip\left( f\lfloor \left[ a_{i},a_{i+1}%
\right] \right) \leq C$. Therefore, $f\in \mathcal{F}$. This shows that $%
\mathcal{F}$ is closed and also equicontinuous. Moreover, for any partition $%
Q$ of $\left[ a_{i},a_{i+1}\right] $, the variation
\begin{equation*}
V_{Q}\left( f\lfloor \left[ a_{i},a_{i+1}\right] \right)
=\lim_{n}V_{Q}\left( \left( f_{n}\right) \lfloor \left[ a_{i},a_{i+1}\right]
\right) \leq \liminf_{n}L\left( \left( f_{n}\right) \lfloor \left[
a_{i},a_{i+1}\right] \right) .
\end{equation*}%
So,

\begin{equation*}
L\left( f\lfloor \left[ a_{i},a_{i+1}\right] \right) =\sup_{Q}V_{Q}\left(
f\lfloor \left[ a_{i},a_{i+1}\right] \right) \leq \liminf_{n}L\left(
f_{n}\lfloor \left[ a_{i},a_{i+1}\right] \right) .
\end{equation*}%
Hence,
\begin{eqnarray*}
L\left( f\right) &=&\sum_{i=0}^{N-1}L\left( f\lfloor _{\left[ a_{i},a_{i+1}%
\right] }\right) \leq \sum_{i=0}^{N-1}\liminf_{n}L\left( f_{n}\lfloor _{%
\left[ a_{i},a_{i+1}\right] }\right) \\
&=&\liminf_{n}L\left( f_{n}\right) .
\end{eqnarray*}
\end{proof}

\begin{proposition}
Suppose $\left( X,J\right) $ is a quasimetric space, and $f\in \mathcal{P}%
_{N}\left( \left[ a,b\right] \text{,}\left( X,J\right) \right) $. If $%
L\left( f\right) =0$, then $f$ is a constant map.
\end{proposition}

\begin{proof}
$L\left( f\right) =0$ implies that $L\left( f\lfloor _{\left[ a_{i},a_{i+1}%
\right] }\right) =0$ for each $i$. Thus, $f$ is a constant on $%
[a_{i},a_{i+1}]$ for each $i$. Since $f$ is continuous, $f$ is a constant on
$\left[ a,b\right] .$
\end{proof}

Since any Lipschitz curve in a metric space has an arc parametrization, by
applying arc parametrizations piecewisely, we also have

\begin{proposition}
(Reparametrization) For any $f\in \mathcal{P}_{N}\left( \left[ a,b\right]
\text{,}\left( X,J\right) \right) $ and $L=L\left( f\right) $, there exists
a homeomorphism $\phi :\left[ 0,L\right] \rightarrow \left[ a,b\right] $ so
that $\gamma =f\circ \phi \in \mathcal{P}_{N}\left( \left[ 0,L\right]
,\left( X,J\right) \right) $ has $\left| \dot{\gamma}\left( t\right) \right|
_{J}=1$ almost everywhere in $\left[ 0,L\right] $.
\end{proposition}

\subsection{The geodesic problem}

\bigskip Let $N$ be a fixed natural number. For any $x,y\in X$, we consider
the geodesic problem
\begin{equation}
\min \{ L\left( f\right) \}  \label{geodesic_problem}
\end{equation}%
among all $f$ in the family
\begin{equation*}
Path_{N}\left( x,y\right) =\left\{ f\in \mathcal{P}_{N}\left( \left[ 0,1%
\right] \text{,}\left( X,J\right) \right) \text{ with }f\left( 0\right)
=x;f\left( 1\right) =y\right\} .
\end{equation*}

Note that, by a linear change of variable, one may replace $\left[ 0,1\right]
$ in $Path_{N}\left( x,y\right) $ by any closed interval $\left[ a,b\right] $
without changing the infimum value in the geodesic problem (\ref%
{geodesic_problem}).

\begin{definition}
Suppose $J$ is a quasimetric on $X$. For any $x,y\in X$, and $N\in \mathbb{N}$%
, define
\begin{equation*}
D_{J}^{\left( N\right) }\left( x,y\right) =\inf \left\{ L_{J}\left( f\right)
:f\in Path_{N}\left( x,y\right) \right\}
\end{equation*}%
whenever $Path_{N}\left( x,y\right) $ is not empty, and set $D_{J}^{\left(
N\right) }\left( x,y\right) =\infty $ when $Path_{N}\left( x,y\right) $ is
empty. Since $D_{J}^{\left( N\right) }\left( x,y\right) $ is a decreasing
function of $N$, we define
\begin{equation*}
D_{J}\left( x,y\right) =\lim_{N\rightarrow \infty }D_{J}^{\left( N\right)
}\left( x,y\right) .
\end{equation*}
\end{definition}

\begin{theorem}
Suppose $J$ is a continuous complete quasimetric on a nonempty set
$X$. For
any\thinspace $\ N\in \mathbb{N}$, and $x,y\in X$, the geodesic problem (\ref%
{geodesic_problem}) admits a solution $f\in Path_{N}\left( x,y\right) $ \
provided that $Path_{N}\left( x,y\right) $ is not empty. So, $L\left(
f\right) =D_{J}^{\left( N\right) }\left( x,y\right) $.
\end{theorem}

\begin{proof}
Suppose $Path_{N}\left( x,y\right) $ is not empty. Let $L=\inf \left\{
L\left( f\right) :f\in Path_{N}\left( x,y\right) \right\} $. Note that for
each $f\in Path_{N}\left( x,y\right) $, we have
\begin{eqnarray*}
J\left( x,y\right) &\leq &\sigma \left( J\right)
^{N-1}\sum_{i=0}^{N-1}J\left( f\left( a_{i}\right) ,f\left( a_{i+1}\right)
\right) \\
&\leq &\sigma \left( J\right) ^{N-1}\sum_{i=0}^{N-1}L\left( f\lfloor _{\left[
a_{i},a_{i+1}\right] }\right) =\sigma \left( J\right) ^{N-1}L\left( f\right)
\text{.}
\end{eqnarray*}%
This implies that if $L=0$, then we have $J\left( x,y\right) =0$. Therefore,
$x=y$ and the constant $f\left( t\right) \equiv x$ is the desired solution.

So, without losing generality, we may assume that $L>0$. Let $\left\{
f_{n}\right\} $ be a length minimizing sequence in $Path_{N}\left(
x,y\right) $ with $L\left( f_{n}\right) \rightarrow L$. Let
\begin{equation*}
P_{f_{n}}=\left\{ 0=a_{0}^{\left( n\right) }<a_{1}^{\left( n\right) }<\cdots
<a_{N}^{\left( n\right) }=1\right\}
\end{equation*}%
be the partition of $\left[ 0,1\right] $, associated with $f_{n}$. By
reparametrization if necessary, we may assume that each $f_{n}$ is Lipschitz
with $Lip\left( f_{n}\right) \leq 1.5L$ on $\left[ a_{i}^{\left( n\right)
},a_{i+1}^{\left( n\right) }\right] $ for each $i=0,\cdots ,N-1$. Then, by
choosing a subsequence if necessary, we may assume that each sequence $%
\left\{ a_{i}^{\left( n\right) }\right\} $ is convergent to some point $%
a_{i} $ as $n\rightarrow \infty $ for each $i=0,1,\cdots ,N$. Using a linear
change of variable, we may assume that for each $i$, $a_{i}^{\left( n\right)
}=a_{i}$ and $Lip\left( f_{n}\right) \leq 2L$ on $\left[ a_{i},a_{i+1}\right]
$. Now, $\left\{ f_{n}\right\} $ is a sequence in the family
\begin{equation*}
\mathcal{F}=\left\{
\begin{array}{c}
f\in C\left( \left[ 0,1\right] ,\left( X,J\right) \right) :\text{ }f\left(
0\right) =x,f\left( 1\right) =y\text{, and }J\text{ is a metric on} \\
f\left( \left[ a_{i},a_{i+1}\right] \right) \text{ and }Lip\left( f\lfloor _{%
\left[ a_{i},a_{i+1}\right] }\right) \leq 2L,\text{ for each }i=0,\cdots ,N-1%
\end{array}%
\right\} .
\end{equation*}%
By lemma \ref{sequentially_compact_family}, $\mathcal{F}$ is a bounded,
closed and equicontinuous subset of $C\left( \left[ 0,1\right] ,\left(
X,J\right) \right) $. By the Ascoli-Arzel\`{a} theorem shown in corollary %
\ref{Ascoli}, a subsequence $\left\{ f_{n_{k}}\right\} $ of $\left\{
f_{n}\right\} $ in $\mathcal{F}$ is uniformly convergent to some $f\in
\mathcal{F}$ with respect to $J_{\infty }$. By the lower semicontinuity of $%
L $ in the family $\mathcal{F}$, we have $L\left( f\right) \leq
\liminf_{k}L\left( f_{n_{k}}\right) =L$. Therefore, $f$ is a length
minimizer in $Path_{N}\left( x,y\right) $.
\end{proof}

Note that each $D_{J}^{\left( N\right) }$ is a semimetric
 \footnote{A function $d:X\times X\rightarrow \lbrack 0,+\infty )$ is a \textit{%
semimetric} on $X$ if $d$ satisfies conditions (\ref{condition_1}),(\ref%
{condition_2}),(\ref{condition_3}) in Definition
\ref{near_metric_def}. So, a semimetric $d$ is not required to
satisfy the triangle inequality.} on $X$ in the
sense that $D_{J}^{\left( N\right) }\left( x,y\right) \geq 0$, $%
D_{J}^{\left( N\right) }\left( x,y\right) =0$ if and only if $x=y$, and $%
D_{J}^{\left( N\right) }\left( x,y\right) =D_{J}^{\left( N\right) }\left(
y,x\right) $. In general, $D_{J}^{\left( N\right) }$ may fail to satisfy the
triangle inequality. Nevertheless, we have
\begin{equation*}
D_{J}^{\left( n+m\right) }\left( x,y\right) \leq D_{J}^{\left( n\right)
}\left( x,z\right) +D_{J}^{\left( m\right) }\left( z,y\right)
\end{equation*}%
for any $m,n$ and $x,y,z\in X$. \ As a result, by letting $N\rightarrow
\infty $, we have

\begin{proposition}
Suppose $J$ is a quasimetric on $X$, then $D_{J}$ is a pseudometric
\footnote{A function $d:X\times X\rightarrow \lbrack 0,+\infty )$ is
a
\textit{pseudometric} on $X$ if $d$ satisfies conditions (\ref{condition_1}%
),(\ref{condition_3}) in Definition \ref{near_metric_def}, and the
triangle inequality $d\left( x,y\right) \leq d\left( x,z\right)
+d\left( z,y\right) $ for any $x,y,z\in X$. But $d\left( x,y\right)
=0$ does not necessarily imply $x=y$.} on $X$.
\end{proposition}

Since $D_{J}$ is a pseudometric, $D_{J}$ is a metric on $X$ if and only if%
\begin{equation*}
D_{J}\left( x,y\right) >0\text{ whenever }x\neq y.
\end{equation*}%
When $D_{J}$ becomes a metric on $X$. This metric is called the \textit{%
intrinsic metric}, or \textit{geodesic distance}, on $X$ induced by
the quasimetric $J$.

\subsection{Examples of metrics induced by quasimetrics}

\bigskip Now, we are interested in cases that $D_{J}$ is indeed a metric on $%
X$.

\subsubsection{Ideal quasimetrics}

Let $J$ be any semimetric on $X$. For any $x,y \in X$, we set
\begin{equation*}
d_J(x,y)
\end{equation*}
to be the infimum of
\begin{equation*}
\sum_{i=1}^{n-1}J\left( x_{i},x_{i+1}\right)
\end{equation*}
over all finitely many points $x_1,\cdots,x_n \in X $ with $x_{1}=x$ and $%
x_{n}=y$.

This $d_{J}$ defines a pseudometric on $X$, but not necessarily a metric on $%
X$.

\begin{example}
For instance, let $X=\left[ 0,1\right] $ and $J\left( x,y\right) =\left|
x-y\right| ^{p}$ for some $p>1$ defines a quasimetric on $X$. Then, for each $%
n$,
\begin{eqnarray*}
d_{J}\left( 0,1\right) &\leq &\sum_{i=0}^{n-1}J\left( \frac{i}{n},\frac{i+1}{%
n}\right) \\
&=&\sum_{i=0}^{n-1}\left( \frac{1}{n}\right) ^{p}=\frac{1}{n^{p-1}}%
\rightarrow 0\text{ as }n\rightarrow \infty \text{.}
\end{eqnarray*}%
Thus, $d_{J}\left( 0,1\right) =0$, but $0\neq 1$. Hence $d_{J}$ is not a
metric on $X$. Also, note that in this example, $Path_{N}\left( x,y\right) $
is empty whenever $x\neq y$. Thus, $D_{J}\left( x,y\right) =\infty $
whenever $x\neq y$.
\end{example}

As in the case of $D_{J}$, $d_{J}$ is a metric on $X$ if and only if%
\begin{equation*}
d_{J}\left( x,y\right) >0\text{ whenever }x\neq y.
\end{equation*}%
Note also that
\begin{equation*}
d_{J}\left( x,y\right) \leq D_{J}^{\left( N\right) }\left( x,y\right)
\end{equation*}%
for each $N$, and thus,%
\begin{equation*}
d_{J}\left( x,y\right) \leq D_{J}\left( x,y\right) \text{.}
\end{equation*}%
Therefore, $d_{J}\left( x,y\right) >0$ will automatically imply $D_{J}\left(
x,y\right) >0$. As a result, we have

\begin{proposition}
\label{d_J-D_J}Suppose $J$ is a quasimetric on $X$.\ If $d_{J}$ is a
metric
on $X$ and $D_{J}\left( x,y\right) <\infty $ for every $x,y\in X$, then $%
D_{J}$ also defines a metric on $X$ .
\end{proposition}

\begin{remark}
When $J$ is indeed a metric on $X$, then both $d_{J}$ and $D_{J}$ are
metrics. In this case, $d_{J}$ is just the metric $J$ itself, while $D_{J}$
is the intrinsic metric induced by $J.$
\end{remark}

In general, by means of definition, we have
\begin{equation*}
d_{J}\left( x,y\right) \leq J\left( x,y\right) \leq \sigma _{\infty }\left(
J\right) d_{J}\left( x,y\right) ,
\end{equation*}%
where $\sigma _{\infty }\left( J\right) $ is defined as in (\ref{ideal}).

Now, suppose $J$ is an ideal quasimetric, then $\sigma _{\infty
}\left( J\right) <\infty $ and $J$ satisfies the condition
\begin{equation*}
J\left( x_{1},x_{n}\right) \leq \sigma _{\infty }\left( J\right)
\sum_{i=1}^{n-1}J\left( x_{i},x_{i+1}\right)
\end{equation*}%
for any finitely many points $\left\{ x_{1},x_{2},\cdots ,x_{n}\right\}
\subset X$. Clearly, we have the following proposition:

\begin{proposition}
Suppose $\left( X,J\right) $ is an ideal quasimetric space. Then for
any $N$ and any $f\in \mathcal{P}_{N}\left( \left[ a,b\right]
\text{,}\left( X,J\right) \right) $, we have
\begin{equation*}
J\left( f\left( a\right) ,f\left( b\right) \right) \leq \sigma _{\infty
}\left( J\right) L\left( f\right) .
\end{equation*}
\end{proposition}

\begin{lemma}
\label{J_ideal}Suppose $J$ is an ideal quasimetric on $X$ . Then,
$d_{J}$ is
a metric on $X$. Moreover, if $D_{J}\left( x,y\right) <\infty $ for every $%
x,y\in X$, then $D_{J}$ also defines a metric on $X$.
\end{lemma}

\begin{proof}
This is simply because when $x\neq y$, $d_{J}\left( x,y\right) \geq \frac{1}{%
\sigma _{\infty }\left( J\right) }J\left( x,y\right) >0$.
\end{proof}

\subsubsection{Perfect quasimetrics}

Here is another kind of quasimetric $J$ which also induces a metric
$D_{J}$.

\begin{definition}
\label{perfect}A quasimetric $J$ on $X$ is a perfect near metric if for any $%
x,y\in X$, the value $D_{J}^{\left( N\right) }\left( x,y\right) $ becomes a
real valued constant $D_{J}\left( x,y\right) $ when $N$ is large enough.
\end{definition}

Since for each $N$, $D_{J}^{\left( N\right) }\left( x,y\right) =0$ if and
only if $x=y$, we have the following theorem.

\begin{proposition}
On a perfect quasimetric space $\left( X,J\right) $, $D_{J}$ defines
a metric on $X$.
\end{proposition}

When $J$ is indeed a metric on $X$, then for each $N$, the metric $%
D_{J}^{\left( N\right) }$ agrees with the intrinsic metric induced
by $J$. Thus, every metric space is automatically a perfect
quasimetric space. In section 4, we will discuss a family of very
important perfect quasimetric spaces, which are not metric spaces.

\begin{theorem}
\label{length_space}Suppose $\left( X,J\right) $ is a perfect
quasimetric space, and the geodesic problem \ref{geodesic_problem}
has solution for $N$ large enough. Then, $\left( X,D_{J}\right) $ is
a length space in the sense that for every $x,y\in X$, there exists
a curve $f:\left[ 0,L\right]
\rightarrow \left( X,D_{J}\right) $ such that $f\left( 0\right) =x$, $%
f\left( L\right) =y$ and
\begin{equation*}
D_{J}\left( f\left( t\right) ,f\left( s\right) \right) =\left| t-s\right|
\end{equation*}%
for every $t,s\in \left[ 0,L\right] $ where $L=D_{J}\left( x,y\right) $.
\end{theorem}

\begin{proof}
For every $x,y\in X$, since $\left( X,J\right) $ is a perfect
quasimetric space, we have $D_{J}^{\left( N\right) }\left(
x,y\right) =D_{J}\left( x,y\right) <\infty $ whenever $N$ is large
enough. Now, for each large enough $N$, there exists a curve
$f:\left[ 0,L\right] \rightarrow \left( X,J\right) $ such that $f$
is the length minimizer in $Path_{N}\left( x,y\right) $ with
$L\left( f\right) =D_{J}^{\left( N\right) }\left( x,y\right)
=D_{J}\left( x,y\right) $. Without losing generality, we may assume
$f$ has its arc parametrization. Now for any $0\leq s<t\leq L$, we
have
\begin{equation*}
D_{J}\left( f\left( s\right) ,f\left( t\right) \right) \leq L\left( f\lfloor
_{\left[ s,t\right] }\right) =\int_{s}^{t}\left| \dot{f}\right| _{J}dt=t-s%
\text{.}
\end{equation*}%
Similarly, $D_{J}\left( f\left( 0\right) ,f\left( s\right) \right) \leq s$
and $D_{J}\left( f\left( t\right) ,f\left( L\right) \right) \leq L-t$. Thus,
we have
\begin{eqnarray*}
L &=&D_{J}\left( x,y\right) \leq D_{J}\left( f\left( 0\right) ,f\left(
s\right) \right) +D_{J}\left( f\left( s\right) ,f\left( t\right) \right)
+D_{J}\left( f\left( t\right) ,f\left( L\right) \right) \\
&\leq &s+\left( t-s\right) +\left( L-t\right) =L.
\end{eqnarray*}%
Therefore,\ all inequalities becomes equalities at every step and\ for any $%
t,s\in \left[ 0,L\right] $, we have $D_{J}\left( f\left( t\right) ,f\left(
s\right) \right) =\left| t-s\right| $ .
\end{proof}

\begin{corollary}
Suppose $J$ is a complete, continuous, perfect quasimetric on $X$. Then, $%
(X,D_J)$ is a length space.
\end{corollary}

The curve $f$ in the theorem \ref{length_space} is called a
\textit{geodesic} from $x$ to $y$ in the perfect quasimetric space
$\left( X,J\right) $.

\section{\protect\bigskip Optimal transport paths as geodesics}

We now begin to introduce a family of both ideal and perfect
quasimetrics on the space of atomic probability measures.

\subsection{\protect\bigskip A family of quasimetrics on the space of atomic
probability measures}

Let $\left( Y,d\right) $ be any metric space. For any $y\in Y$, let $\delta
_{y}$ be the Dirac measure centered at $y$. An atomic probability measure in
$Y$ is in the form of
\begin{equation*}
\sum_{i=1}^{m}a_{i}\delta _{y_{i}}
\end{equation*}%
with distinct points $y_{i}\in Y$, and $a_{i}>0$ with $\sum_{i=1}^{m}a_{i}=1$%
.

Given two atomic probability measures
\begin{equation}
\mathbf{a}=\sum_{i=1}^{m}a_{i}\delta _{x_{i}}\text{ and }\mathbf{b}%
=\sum_{j=1}^{n}b_{j}\delta _{y_{j}}  \label{atomic_measures}
\end{equation}%
in $Y$, a \textit{transport plan} from $\mathbf{a}$ to $\mathbf{b}$ is an
atomic probability measure
\begin{equation}
\gamma =\sum_{i=1}^{m}\sum_{j=1}^{n}\gamma _{ij}\delta _{\left(
x_{i},y_{j}\right) }  \label{transport_plan}
\end{equation}%
in the product space $Y\times Y$ such that
\begin{equation}
\sum_{i=1}^{m}\gamma _{ij}=b_{j}\text{ and }\sum_{j=1}^{n}\gamma _{ij}=a_{i}
\label{margins}
\end{equation}%
for each $i$ and $j$. Let $Plan\left( \mathbf{a},\mathbf{b}\right) $ be the
space of all transport plans from $\mathbf{a}$ to $\mathbf{b}$.

For any $\alpha <1$, we now introduce the functional $H_{\alpha }$ on
transport plans. For any atomic probability measure $\gamma $ in $Y\times Y$
of the form (\ref{transport_plan}), we define
\begin{equation*}
H_{\alpha }\left( \gamma \right) :=\sum_{i=1}^{m}\sum_{j=1}^{n}\left( \gamma
_{ij}\right) ^{\alpha }d\left( x_{i},y_{j}\right) ,
\end{equation*}%
where $d$ is the given metric on $Y$.

Using $H_{\alpha }$, we may define

\begin{definition}
For any two atomic probability measures $\mathbf{a},\mathbf{b}$ on $Y$, and $%
\alpha <1$, define
\begin{equation*}
J_{\alpha }\left( \mathbf{a},\mathbf{b}\right) :=\min \left\{ H_{\alpha
}\left( \gamma \right) :\gamma \in Plan\left( \mathbf{a},\mathbf{b}\right)
\right\} .
\end{equation*}
\end{definition}

For any given natural number $N\in \mathbb{N}$ , let $\mathcal{A}_{N}(Y)$ be
the space of all atomic probability measures
\begin{equation*}
\sum_{i=1}^{m}a_{i}\delta _{x_{i}}
\end{equation*}%
on $Y$ with $m\leq N$, and $\mathcal{A}\left( Y\right) =\bigcup_{N}\mathcal{A%
}_{N}\left( Y\right) $ be the space of all atomic probability measures on $Y$%
.

\begin{proposition}
$J_{\alpha }$ defines a quasimetric on $\mathcal{A}_{N}\left(
Y\right) $ with $\sigma \left( J_{\alpha }\right) \leq N^{1-\alpha
}$.
\end{proposition}

\begin{proof}
For any $\mathbf{a},\mathbf{b}\in \mathcal{A}_{N}\left( Y\right) $ in the
form of (\ref{atomic_measures}), clearly $J_{\alpha }\left( \mathbf{a},%
\mathbf{b}\right) \geq 0$ and $J_{\alpha }\left( \mathbf{a},\mathbf{b}%
\right) =J_{\alpha }\left( \mathbf{b},\mathbf{a}\right) $.

If $J_{\alpha }\left( \mathbf{a},\mathbf{b}\right) =0$, then there exists a $%
\gamma \in Plan\left( \mathbf{a},\mathbf{b}\right) $ such that $H_{\alpha
}\left( \gamma \right) =0$. Thus, $d\left( x_{i},y_{j}\right) =0$ whenever $%
\gamma _{ij}\neq 0$. Since $\left\{ y_{j}\right\} $'s are distinct, at most
one of $\gamma _{ij}$ can be nonzero for each $i$. On the other hand, by (%
\ref{margins}), at least one of $\gamma _{ij}$ must be nonzero for each $i$.
Therefore, for each $i$, there is a unique $j=\sigma \left( i\right) $ such
that $x_{i}=y_{j}$ and $\gamma _{ij}=a_{i}=b_{j}$. This shows that $\mathbf{a%
}=\mathbf{b}$.

Now, we prove that $J$ satisfies the relaxed triangle inequality as in
condition \ref{condition_4} in Definition \ref{near_metric_def}. Indeed, for
any
\begin{equation*}
\mathbf{a}=\sum_{i=1}^{m}a_{i}\delta _{x_{i}}\text{, }\mathbf{b}%
=\sum_{j=1}^{n}b_{j}\delta _{y_{j}}\text{ and }\mathbf{c}%
=\sum_{k=1}^{h}c_{k}\delta _{z_{k}}
\end{equation*}%
in $\mathcal{A}_{N}\left( Y\right) $, and any
\begin{equation*}
u_{\mathbf{a}}^{\mathbf{c}}=\sum_{i=1}^{m}\sum_{k=1}^{h}u_{ik}\delta
_{\left( x_{i},z_{k}\right) }\in Path\left( \mathbf{a},\mathbf{c}\right)
\text{ and }\tau _{\mathbf{c}}^{\mathbf{b}}=\sum_{j=1}^{n}\sum_{k=1}^{h}\tau
_{kj}\delta _{\left( z_{k},y_{j}\right) }\in Path\left( \mathbf{c},\mathbf{b}%
\right) ,
\end{equation*}%
we denote
\begin{equation*}
\gamma _{ij}=\sum_{k=1}^{h}\frac{u_{ik}\tau _{kj}}{c_{k}}
\end{equation*}%
for each $i,j$. Note that%
\begin{equation*}
\sum_{i=1}^{m}\gamma _{ij}=\sum_{i=1}^{m}\left( \sum_{k=1}^{h}\frac{%
u_{ik}\tau _{kj}}{c_{k}}\right) =\sum_{k=1}^{h}\left( \sum_{i=1}^{m}\frac{%
u_{ik}\tau _{kj}}{c_{k}}\right) =\sum_{k=1}^{h}\tau _{kj}=b_{j}\text{ }
\end{equation*}%
and similarly $\sum_{j}\gamma _{ij}=a_{i}.$Therefore, we find a transport
plan%
\begin{equation*}
\gamma =\sum_{i=1}^{m}\sum_{j=1}^{n}\gamma _{ij}\delta _{\left(
x_{i},y_{j}\right) }\in Plan\left( \mathbf{a},\mathbf{b}\right) .
\end{equation*}%
We now want to show
\begin{equation*}
H_{\alpha }\left( \gamma \right) \leq N\left( H_{\alpha }\left( u_{\mathbf{a}%
}^{\mathbf{c}}\right) +H_{\alpha }\left( \tau _{\mathbf{c}}^{\mathbf{b}%
}\right) \right) .
\end{equation*}%
Indeed,
\begin{eqnarray*}
&&H_{\alpha }\left( \gamma \right) =\sum_{i=1}^{m}\sum_{j=1}^{n}\left(
\gamma _{ij}\right) ^{\alpha }d\left( x_{i},y_{j}\right)
=\sum_{i=1}^{m}\sum_{j=1}^{n}\left( \sum_{k=1}^{h}\frac{u_{ik}\tau _{kj}}{%
c_{k}}\right) ^{\alpha }d\left( x_{i},y_{j}\right)  \\
&\leq &\sum_{i=1}^{m}\sum_{j=1}^{n}\sum_{k=1}^{h}\left( \frac{u_{ik}\tau
_{kj}}{c_{k}}\right) ^{\alpha }\left( d\left( x_{i},z_{k}\right) +d\left(
z_{k},y_{j}\right) \right) ,\text{ because }\alpha <1 \\
&=&\sum_{i=1}^{m}\sum_{k=1}^{h}\left( \sum_{j=1}^{n}\left( \frac{u_{ik}\tau
_{kj}}{c_{k}}\right) ^{\alpha }\right) d\left( x_{i},z_{k}\right)
+\sum_{j=1}^{n}\sum_{k=1}^{h}\left( \sum_{i=1}^{m}\left( \frac{u_{ik}\tau
_{kj}}{c_{k}}\right) ^{\alpha }\right) d\left( z_{k},y_{j}\right)  \\
&\leq &N^{1-\alpha }\left( \sum_{i=1}^{m}\sum_{k=1}^{h}\left( u_{ik}\right)
^{\alpha }d\left( x_{i},z_{k}\right) +\sum_{j=1}^{n}\sum_{k=1}^{h}\left(
\tau _{kj}\right) ^{\alpha }d\left( z_{k},y_{j}\right) \right)  \\
&=&N^{1-\alpha }\left( H_{\alpha }\left( u_{\mathbf{a}}^{\mathbf{c}}\right)
+H_{\alpha }\left( \tau _{\mathbf{c}}^{\mathbf{b}}\right) \right) ,
\end{eqnarray*}%
where the 2nd inequality follows from the inequality $\sum_{i=1}^{N}\left(
t_{i}\right) ^{\alpha }\leq N^{1-\alpha }\left( \sum_{i=1}^{N}t_{i}\right)
^{\alpha }$. Therefore, by taking infimum, we have%
\begin{equation*}
J_{\alpha }\left( \mathbf{a},\mathbf{b}\right) \leq N^{1-\alpha }\left(
J_{\alpha }\left( \mathbf{a},\mathbf{c}\right) +J_{\alpha }\left( \mathbf{c},%
\mathbf{b}\right) \right) .
\end{equation*}
\end{proof}

\begin{proposition}
 Suppose $\left( Y,d\right) $ is a complete metric space. Then, $%
J_{\alpha }$ is a complete quasimetric on $\mathcal{A}_{N}\left(
Y\right) $.
\end{proposition}

\begin{proof}
Let $\left\{ \mathbf{a}_{n}\right\} $ be any Cauchy sequence in $\mathcal{A}%
_{N}\left( Y\right) $. Then, for any $\epsilon >0$, there exists a
natural number $\tilde{N}$, such that
\[
J_{\alpha }\left( \mathbf{a}_{n},\mathbf{a}_{m}\right) \leq \epsilon
\]%
whenever $n,m\geq \tilde{N}$. Note that each atomic probability measure $%
\mathbf{a}_{n}$ may be expressed as
\[
\mathbf{a}_{n}=\sum_{i=1}^{N}a_{i}^{\left( n\right) }\delta
_{x_{i}^{\left( n\right) }}
\]%
for some $a_{i}^{\left( n\right) }\geq 0$,
$\sum_{i=1}^{N}a_{i}^{\left( n\right) }=1$ and $x_{i}^{\left(
n\right) }\in Y$.

Now, let $\gamma ^{\left( n,m\right) }$ be an $H_{\alpha }$ minimizer in $%
Plan\left( \mathbf{a}_{n},\mathbf{a}_{m}\right) $ with
\[
J_{\alpha }\left( \mathbf{a}_{n},\mathbf{a}_{m}\right) =H_{\alpha
}\left( \gamma ^{\left( n,m\right) }\right) \text{.}
\]%
This transport plan $\gamma ^{\left( n,m\right) }$ is expressed as
\[
\gamma ^{\left( n,m\right) }=\sum_{i,j=1}^{N}\gamma _{ij}^{\left(
n,m\right) }\delta _{\left( x_{i}^{\left( n\right) },x_{j}^{\left(
m\right) }\right) }
\]%
for some $\gamma _{ij}^{\left( n,m\right) }\geq 0$ with $\sum_{i=1}^{N}%
\gamma _{ij}^{\left( n,m\right) }=a_{j}^{\left( m\right) }$ and $%
\sum_{j=1}^{N}\gamma _{ij}^{\left( n,m\right) }=a_{i}^{\left(
n\right) }$ for all $i,j=1,2,\cdots ,N$.

By picking a subsequence if necessary, without lossing generality,
we may use the diagonal argument and assume that for all
$i,j=1,2,\cdots ,N$ and
all $n\geq \tilde{N}$%
\[
\gamma _{ij}^{\left( n,m\right) }\rightarrow \gamma _{ij}^{\left(
n\right) }
\]%
as $m\rightarrow \infty $. Then, for each $i,j$ and each $n\geq
\tilde{N}$, we have
\begin{equation}
\sum_{i=1}^{N}\gamma _{ij}^{\left( n\right) }=\lim_{m\rightarrow
\infty }\sum_{i=1}^{N}\gamma _{ij}^{\left( n,m\right)
}=\lim_{m\rightarrow \infty }a_{j}^{\left( m\right) }\text{ and
}\sum_{j=1}^{N}\gamma _{ij}^{\left( n\right) }=a_{i}^{\left(
n\right) }.  \label{margin_limit}
\end{equation}%
Let
\[
a_{j}=\lim_{m\rightarrow \infty }a_{j}^{\left( m\right) }
\]%
for each $j$. If $a_{j}>0$, then by (\ref{margin_limit}), there
exists an $i$ such that $\gamma _{ij}^{\left( n\right) }>0$. So
\[
d\left( x_{i}^{\left( n\right) },x_{j}^{\left( m\right) }\right) \leq \frac{%
H_{\alpha }\left( \gamma ^{\left( n,m\right) }\right) }{\left[
\gamma _{ij}^{\left( n,m\right) }\right] ^{\alpha }}=\frac{J_{\alpha
}\left( \mathbf{a}_{n},\mathbf{a}_{m}\right) }{\left[ \gamma
_{ij}^{\left( n,m\right) }\right] ^{\alpha }}
\]%
which implies that
\[
\left\{ x_{j}^{\left( m\right) }\right\} _{m=1}^{\infty }
\]%
is a Cauchy sequence in the complete metric space $\left( Y,d\right)
$. Thus, $x_{j}^{\left( m\right) }\rightarrow x_{j}$ as
$m\rightarrow \infty $ for some $x_{j}\in Y$.

Let
\[
\mathbf{a}=\sum_{a_{j}>0}a_{j}\delta _{x_{j}}\in
\mathcal{A}_{N}\left( Y\right)
\]%
and for each $n\geq \tilde{N}$, let
\[
\gamma ^{\left( n\right) }=\sum_{ij}\gamma _{ij}^{\left( n\right)
}\delta _{\left( x_{i}^{\left( n\right) },x_{j}\right) }.
\]%
Then, $\gamma ^{\left( n\right) }\in Plan\left( \mathbf{a}_{n},\mathbf{a}%
\right) $ $\ $and
\begin{eqnarray*}
J_{\alpha }\left( \mathbf{a}_{n},\mathbf{a}\right)  &\leq &H_{\alpha
}\left(
\gamma ^{\left( n\right) }\right)  \\
&=&\sum_{ij}\left[ \gamma _{ij}^{\left( n\right) }\right] ^{\alpha
}d\left(
x_{i}^{\left( n\right) },x_{j}\right)  \\
&=&\lim_{m\rightarrow \infty }\sum_{ij}\left[ \gamma _{ij}^{\left(
n,m\right) }\right] ^{\alpha }d\left( x_{i}^{\left( n\right)
},x_{j}^{\left(
m\right) }\right)  \\
&=&\lim_{m\rightarrow \infty }J_{\alpha }\left( \mathbf{a}_{n},\mathbf{a}%
_{m}\right) \leq \epsilon .
\end{eqnarray*}%
Therefore, $\left\{ \mathbf{a}_{n}\right\} $ is (subsequentially)
convergent
to $\mathbf{a}$ in $\left( \emph{A}_{N}\left( Y\right) ,J_{\alpha }\right) $%
. This shows that $\mathcal{A}_{N}\left( Y\right) $ is complete with
respect to the quasimetric $J_{\alpha }$.
\end{proof}
Note that, in general, $J_{\alpha }$ may fail to be a metric on $\mathcal{A}%
_{N}\left( Y\right) $ as demonstrated in the following example.

\begin{example}
For any $\alpha <1$, let $y$ be a positive real number. Then, we consider
three atomic measures in $Y=\mathbb{R}^{2}:$
\begin{equation*}
\mathbf{a}=\frac{1}{2}\delta _{\left( -1,y+1\right) }+\frac{1}{2}\delta
_{\left( 1,y+1\right) },\mathbf{b}=\delta _{\left( 0,0\right) }\text{ and }%
\mathbf{c}=\delta _{\left( 0,y\right) }\text{.}
\end{equation*}%
Then,
\begin{eqnarray*}
&&J_{\alpha }\left( \mathbf{a},\mathbf{c}\right) +J_{\alpha }\left( \mathbf{c%
},\mathbf{b}\right) -J_{\alpha }\left( \mathbf{a},\mathbf{b}\right) \\
&=&2\left( \frac{1}{2}\right) ^{\alpha }\sqrt{2}+y-2\left( \frac{1}{2}%
\right) ^{\alpha }\sqrt{1+\left( y+1\right) ^{2}}<0
\end{eqnarray*}%
whenever $y$ is large enough. Thus,\thinspace $J_{\alpha }$ does not satisfy
the triangle inequality.
\end{example}

\subsection{Optimal transport paths between atomic probability measures}

Now, we want to show that the quasimetric $J_{\alpha }$ is both
ideal and perfect. To achieve these results, we first recall some
concepts about
optimal transport paths between probability measures as studied in \cite%
{xia1}.

Let $\mathbf{a}$ and $\mathbf{b}$ be two fixed atomic probability measures
in the form of (\ref{atomic_measures}).

\begin{definition}
A \textit{transport path} from $\mathbf{a}$ to $\mathbf{b}$ is a weighted
directed graph $G$ consists of a vertex set $V\left( G\right) $, a directed
edge set $E\left( G\right) $ and a weight function
\begin{equation*}
w:E\left( G\right) \rightarrow \left( 0,+\infty \right)
\end{equation*}%
such that $\left\{ x_{1,}x_{2,\cdots ,}x_{k}\right\} \cup \left\{
y_{1},y_{2},\cdots ,y_{l}\right\} \subset V\left( G\right) $ and for any
vertex $v\in V\left( G\right) ,$%
\begin{equation}
\sum_{\substack{ e\in E\left( G\right)  \\ e^{-}=v}}w\left( e\right) =\sum
_{\substack{ e\in E\left( G\right)  \\ e^{+}=v}}w\left( e\right) +\left\{
\begin{array}{cc}
a_{i}, & \text{if }v=x_{i}\text{ for some }i=1,\cdots ,k \\
-b_{j}, & \text{if }v=y_{j}\text{ for some }j=1,\cdots ,l \\
0, & \text{otherwise}%
\end{array}%
\right.  \label{balance_equation}
\end{equation}%
where $e^{-}$ and $e^{+}$denotes the starting and ending endpoints of each
edge $e\in E\left( G\right) $.
\end{definition}

\begin{remark}
The balance equation (\ref{balance_equation}) simply means that the total
mass flows into $v$ equals to the total mass flows out of $v$. When $G$ is
viewed as a polyhedral chain or current, (\ref{balance_equation}) can be
simply expressed as
\begin{equation*}
\partial G=\mathbf{b}-\mathbf{a}\text{.}
\end{equation*}%
Also, when $G$ is viewed as a vector valued measure, the balance equation is
simply
\begin{equation*}
div\left( G\right) =\mathbf{a}-\mathbf{b}
\end{equation*}%
in the sense of distributions.
\end{remark}

Let $Path(\mathbf{a},\mathbf{b})$ be the space of all transport paths from $%
\mathbf{a}$ to $\mathbf{b}$.

\begin{definition}
For any $\alpha \leq 1$, and any $G\in Path(\mathbf{a},\mathbf{b})$, define
\begin{equation*}
\mathbf{M}_{\alpha }\left( G\right) :=\sum_{e\in E\left( G\right) }w\left(
e\right) ^{\alpha }length\left( e\right) .
\end{equation*}
\end{definition}

\begin{remark}
In \cite{xia1}, the parameter $\alpha$ was restricted in $[0,1]$. Later, the
author observed that $\alpha<0$ is also very interesting, and related to
studying the dimension of fractals. So, negative $\alpha$ is also allowed
here.
\end{remark}

We first recite two lemmas that were proved in \cite[Proposition 2.1]{xia1}
and \cite[Definition 7.1 and Lemma 7.1]{xia1} respectively.

\begin{lemma}
\label{no_loops} For any transport path $G\in Path\left( \mathbf{a},\mathbf{b%
}\right) $, there exists another transport path $\tilde{G}\in Path\left(
\mathbf{a},\mathbf{b}\right) $ such that
\begin{equation*}
\mathbf{M}_{\alpha }\left( \tilde{G}\right) \leq \mathbf{M}_{\alpha }\left(
G\right) ,
\end{equation*}%
the set of vertices $V\left( \tilde{G}\right) \subset V\left( G\right) $ and
$\tilde{G}$ contains no cycles.
\end{lemma}

Here, a weighted directed graph $G=\left\{ V\left( G\right) ,E\left(
G\right) ,W:E\left( G\right) \rightarrow (0,1]\right\} $ contains a \textit{%
cycle} if for some $k\geq 3$, there exists a list of distinct vertices $%
\left\{ v_{1},v_{2},\cdots ,v_{k}\right\} $ in $V\left( G\right) $ such that
for each $i=1,\cdots ,k$, either\ the segment $\left[ v_{i},v_{i+1}\right] $
or $\left[ v_{i+1},v_{i}\right] $ is a directed edge in $E(G)$, with the
agreement that $v_{k+1}=v_{1}$. When a directed graph $G$ contains no
cycles, it becomes a directed tree.

\begin{lemma}
\label{decomposition}For any transport path $G\in Path\left( \mathbf{a},%
\mathbf{b}\right) $ containing no cycles, there exists

\begin{enumerate}
\item an $m\times n$ real matrix
\begin{equation*}
u=\left( u_{ij}\right) \text{ with }
\end{equation*}%
\begin{equation*}
u_{ij}\geq 0,\sum_{i=1}^{m}u_{ij}=b_{j}\text{,}\sum_{j=1}^{n}u_{ij}=a_{i}%
\text{ for each\thinspace }i,j\text{ and}\sum_{i=1}^{m}%
\sum_{j=1}^{n}u_{ij}=1,
\end{equation*}

\item and an $m\times n$ matrix
\begin{equation*}
g=\left( g_{ij}\right)
\end{equation*}%
with each $g_{ij}$ is either $0$ or an oriented polyhedral curve $g_{ij}$
from $x_{i}$ to $y_{j}$,
\end{enumerate}

such that
\begin{equation*}
G=\sum_{i,j}u_{ij}g_{ij}
\end{equation*}%
as real coefficients polyhedral chains.
\end{lemma}

By means of lemma \ref{no_loops}, it is easy to see that for each $\alpha
\leq 1$, there exists an optimal transport path in $Path\left( \mathbf{a},%
\mathbf{b}\right) $ which minimizes the cost functional $\mathbf{M}_{\alpha
} $.

For the sake of visualization we provide some numerical simulations (see the
forthcoming paper \cite{xia6}) for different values of $\alpha$.

\begin{example}
Let $\left\{ x_{i}\right\} $ be 50 random points in the square $%
\left[ 0,1\right] \times \left[ 0,1\right] $. Then, $\left\{ x_{i}\right\} $
determines an atomic probability measure
\begin{equation*}
\mathbf{a}=\sum_{i=1}^{50}\frac{1}{50}\delta _{x_{i}}.
\end{equation*}%
Let $\mathbf{b}=\delta _{O}$ where $O=\left( 0,0\right) $ is the origin.
Then an optimal transport path from $\mathbf{a}$ to $\mathbf{b}$ looks like
the following figures with $\alpha=1, 0.75, 0.5$ and $0.25$ respectively:
\end{example}
\includegraphics[height=1.75in]{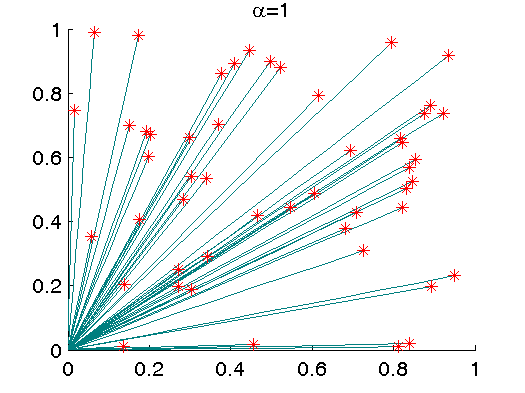}%
\includegraphics[height=1.75in]{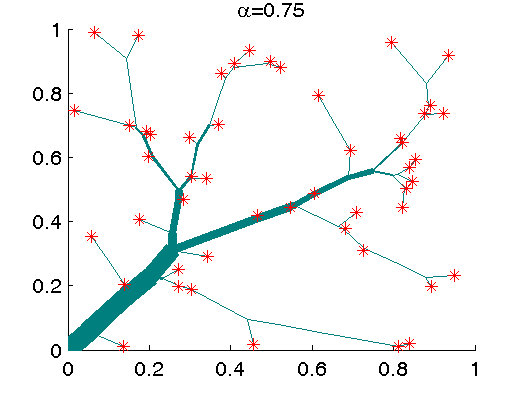}

\includegraphics[height=1.75in]{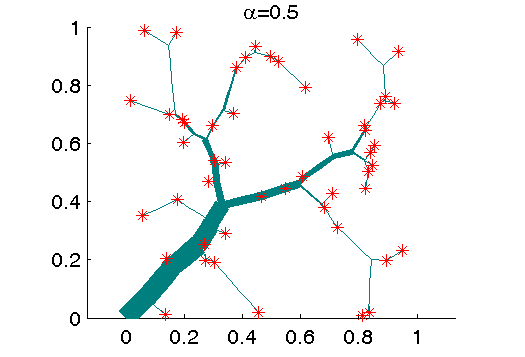}%
\includegraphics[height=1.75in]{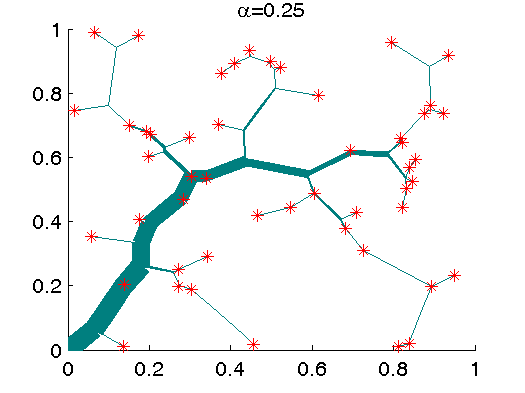}

\begin{example}
Let $\left\{ x_{i}\right\} $ be 100 random points in the rectangle
$\left[ -2.5,2.5\right] \times \left[ 0,1\right] $. Then, $\left\{
x_{i}\right\} $
determines an atomic probability measure \[\mathbf{a}=\sum_{i=1}^{100}\frac{1%
}{100}\delta _{x_{i}}.\] Let $\mathbf{b}=\delta _{O}$ where
$O=\left( 0,0\right) $ is the origin, and let $\alpha =0.85$. Then
an optimal transport path from $\mathbf{a}$ to $\mathbf{b}$ looks
like the following figure.
\end{example}
\begin{center}
\includegraphics[width=5in]{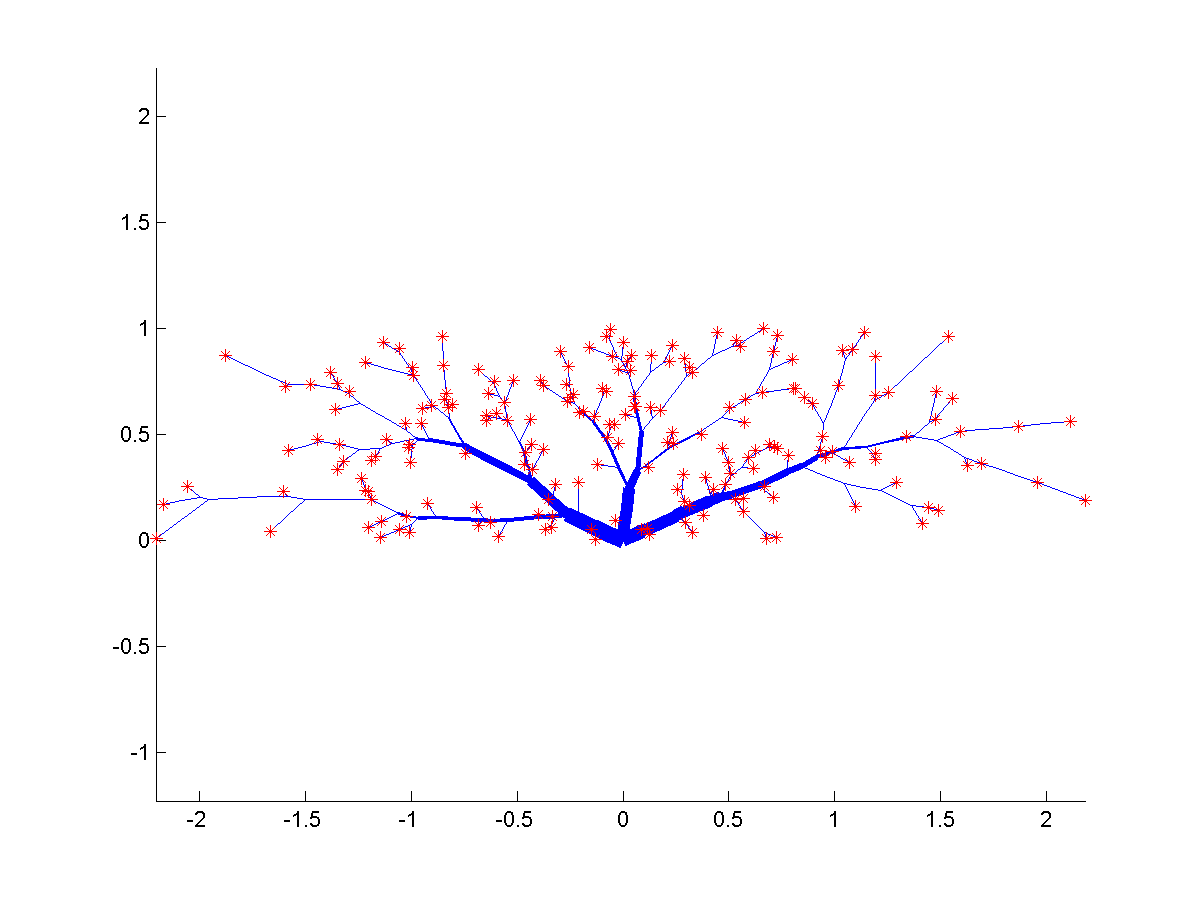}
\end{center}

\subsection{Relation between optimal transport paths and quasimetrics $J_{%
\protect\alpha }$}

We now start to investigate the relationship between optimal
transport path and the quasimetric $J_{\alpha }$ on
$\mathcal{A}_{N}\left( Y\right) $. We
first observe that any transport plan $\gamma \in Plan\left( \mathbf{a},%
\mathbf{b}\right) $ in the form of (\ref{transport_plan}) determines a
transport path $G_{\gamma }\in Path\left( \mathbf{a},\mathbf{b}\right) $.
Indeed, we consider the weighted directed graph $G_{\gamma }$ with
\begin{eqnarray*}
V\left( G_{\gamma }\right) &=&\left\{ x_{1},\cdots ,x_{m},y_{1},\cdots
,y_{n}\right\} , \\
E\left( G_{\gamma }\right) &=&\left\{ \text{a pair }\left[ x_{i},y_{j}\right]
\text{ if }\gamma _{ij}\neq 0\right\} ,
\end{eqnarray*}%
and setting the weight $W\left( \left[ x_{i},y_{j}\right] \right) =\gamma
_{ij}$ for each $i,j$ with $\gamma _{ij}\neq 0$. Moreover,
\begin{equation*}
\mathbf{M}_{\alpha }\left( G_{\gamma }\right) =\sum_{e\in E\left( G_{\gamma
}\right) }w\left( e\right) ^{\alpha }length\left( e\right) =\sum_{i,j}\left(
\gamma _{ij}\right) ^{\alpha }d\left( x_{i},y_{j}\right) =H_{\alpha }\left(
\gamma \right) .
\end{equation*}

\begin{proposition}
\label{M_alpha_J_alpha}For any $\mathbf{a}^{\left( 1\right) },\mathbf{a}%
^{\left( 2\right) },\cdots ,\mathbf{a}^{\left( k\right) }\in \mathcal{A}%
\left( Y\right) $, there exists a transport path $G\in Path\left( \mathbf{a}%
^{\left( 1\right) },\mathbf{a}^{\left( k\right) }\right) $ such that
\begin{equation*}
\mathbf{M}_{\alpha }\left( G\right) \leq \sum_{i=1}^{k-1}J_{\alpha }\left(
\mathbf{a}^{\left( i\right) },\mathbf{a}^{\left( i+1\right) }\right)
\end{equation*}%
and $G$ contains no cycles.
\end{proposition}

\begin{proof}
Let $\gamma _{i}$ be an optimal transport path from $\mathbf{a}^{\left(
i\right) }$ to $\mathbf{a}^{\left( i+1\right) }$, for each $i=1,2,\cdots
,k-1 $. Each $\gamma _{i}$ determines a transport path $G_{\gamma _{i}}\in
Path\left( \mathbf{a}^{\left( i\right) },\mathbf{a}^{\left( i+1\right)
}\right) $ as above. Then, viewed as real coefficients polyhedral chains,
\begin{equation*}
G=\sum_{i=1}^{k-1}G_{\gamma _{i}}
\end{equation*}%
is a transport path from $\mathbf{a}^{\left( 1\right) }$ to $\mathbf{a}%
^{\left( k\right) }$. Moreover, we have
\begin{equation*}
\mathbf{M}_{\alpha }\left( G\right) \leq \sum_{i=1}^{k-1}\mathbf{M}_{\alpha
}\left( G_{\gamma _{i}}\right) =\sum_{i=1}^{k-1}H_{\alpha }\left( \gamma
_{i}\right) =\sum_{i=1}^{k-1}J_{\alpha }\left( \mathbf{a}^{\left( i\right) },%
\mathbf{a}^{\left( i+1\right) }\right) .
\end{equation*}%
By lemma \ref{no_loops}, there exists a transport path $\tilde{G}$ from $%
\mathbf{a}^{\left( 1\right) }$ to $\mathbf{a}^{\left( k\right) }$ such that $%
\tilde{G}$ contains no cycles, $V\left( \tilde{G}\right) \subset V\left(
G\right) $, and
\begin{equation*}
\mathbf{M}_{\alpha }\left( \tilde{G}\right) \leq \mathbf{M}_{\alpha }\left(
G\right) \leq \sum_{i=1}^{k-1}J_{\alpha }\left( \mathbf{a}^{\left( i\right)
},\mathbf{a}^{\left( i+1\right) }\right) .
\end{equation*}
\end{proof}

\begin{theorem}
$J_{\alpha }$ is an ideal quasimetric on $\mathcal{A}_{N}\left(
Y\right) $ with $\sigma _{\infty }\left( J_{\alpha }\right) \leq
N^{2\left( 1-\alpha \right) }$.
\end{theorem}

\begin{proof}
For any $k\in \mathbb{N}$ and any points $\left\{ \mathbf{a}^{\left(
1\right) },\mathbf{a}^{\left( 2\right) },\cdots ,\mathbf{a}^{\left( k\right)
}\right\} \subset \mathcal{A}_{N}\left( Y\right) $, by proposition \ref%
{M_alpha_J_alpha}, there exists a transport path $G\in Path\left( \mathbf{a}%
^{\left( 1\right) },\mathbf{a}^{\left( k\right) }\right) $ such that
\begin{equation*}
\mathbf{M}_{\alpha }\left( G\right) \leq \sum_{i=1}^{k-1}J_{\alpha }\left(
\mathbf{a}^{\left( i\right) },\mathbf{a}^{\left( i+1\right) }\right)
\end{equation*}%
and $G$ contains no cycles. Moreover, by lemma \ref{decomposition}, there
exists a matrix $\left( u_{ij}\right) $ of real numbers and a matric $\left(
g_{ij}\right) $ of polyhedral curves such that
\begin{equation*}
G=\sum_{ij}u_{ij}g_{ij}
\end{equation*}%
as real coefficients polyhedral chains. Let
\begin{equation*}
\gamma =\sum_{ij}u_{ij}\delta _{\left( x_{i},y_{j}\right) }
\end{equation*}%
be any transport plan in $Plan\left( \mathbf{a}^{\left( 1\right) },\mathbf{a}%
^{\left( k\right) }\right) $. Then,
\begin{eqnarray*}
H_{\alpha }\left( \gamma \right) &=&\sum_{ij}\left( u_{ij}\right) ^{\alpha
}d\left( x_{i},y_{j}\right) \leq \sum_{ij}\left( u_{ij}\right) ^{\alpha
}length\left( g_{ij}\right) \\
&=&\sum_{e\in E\left( G\right) }\left( \sum_{g_{ij}\text{ contains }e}\left(
u_{ij}\right) ^{\alpha }\right) length\left( e\right) \\
&\leq &\sum_{e}\left( N^{2\left( 1-\alpha \right) }\left( \sum_{g_{ij}\text{
contains }e}u_{ij}\right) ^{\alpha }\right) length\left( e\right) \\
&=&N^{2\left( 1-\alpha \right) }\sum_{e\in E\left( G\right) }\left( w\left(
e\right) \right) ^{\alpha }length\left( e\right) \\
&=&N^{2\left( 1-\alpha \right) }\mathbf{M}_{\alpha }\left( G\right) \leq
N^{2\left( 1-\alpha \right) }\sum_{i=1}^{k-1}J_{\alpha }\left( \mathbf{a}%
^{\left( i\right) },\mathbf{a}^{\left( i+1\right) }\right) .
\end{eqnarray*}%
Therefore,
\begin{equation*}
J_{\alpha }\left( \mathbf{a}^{\left( 1\right) },\mathbf{a}^{\left( k\right)
}\right) \leq N^{2\left( 1-\alpha \right) }\sum_{i=1}^{k-1}J_{\alpha }\left(
\mathbf{a}^{\left( i\right) },\mathbf{a}^{\left( i+1\right) }\right)
\end{equation*}%
and thus $J_{\alpha }$ is an ideal quasimetric on
$\mathcal{A}_{N}\left( Y\right) $ with $\sigma _{\infty }\left(
J_{\alpha }\right) \leq N^{2\left( 1-\alpha \right) }$.
\end{proof}

Suppose $\left( Y,d\right) $ is a geodesic metric space. That is, for any $%
x,y\in Y$, there exists a Lipschitz curve $\Gamma _{x,y}:\left[ 0,1\right]
\rightarrow \left( Y,d\right) $ with $\Gamma _{x,y}\left( 0\right) =x$, $%
\Gamma _{x,y}\left( 1\right) =y$ and length $L\left( \Gamma _{x,y}\right)
=d\left( x,y\right) $.

\begin{lemma}
\label{path2curve}Suppose $\left( Y,d\right) $ is a geodesic metric space.
Let $G\in Path\left( \mathbf{a,b}\right) $ for some $\mathbf{a,b\in }%
\mathcal{A}_{N}\left( Y\right) $. If each edge of $G$ is a geodesic curve
between its endpoints in the metric space $Y$, then there exists a piecewise
metric Lipschitz curve $g\in \mathcal{P}_{N_{G}}\left( \left[ 0,1\right]
\text{,}\left( \mathcal{A}_{N}\left( Y\right) ,J_{\alpha }\right) \right) $
such that
\begin{equation*}
L_{J_{\alpha }}\left( g\right) =\mathbf{M}_{\alpha }\left( G\right) ,
\end{equation*}%
where $N_{G}$ is total number of edges in the graph $G$.
\end{lemma}

\begin{proof}
We may prove it using the mathematical induction on $N_{G}$. When $N_{G}=1$,
$G$ itself is a geodesic in $Y$. Then, it is clearly true in this case. \
Now, assume $N_{G}>1$. Pick an edge $e$ of $G$ with its starting endpoint $%
e^{-}$ being a vertex in $\mathbf{a}$. Let
\begin{equation*}
\mathbf{\tilde{a}}=\mathbf{a}+w\left( e\right) \left( \mathbf{\delta }%
_{e^{+}}-\delta _{e^{-}}\right) ,
\end{equation*}%
where $e^{+}$ is the targeting endpoint of the directed edge $e$, and $%
w\left( e\right) $ is the associated weight on $e$. Removing edge $e$ from $%
G $, we get another transport path $\tilde{G}\in Path\left( \mathbf{\tilde{a}%
,b}\right) .$ Then, $N_{\tilde{G}}=N_{G}-1\geq 1$. By the principle of the
mathematical induction, we may assume that $\tilde{G}$ corresponds to a
piecewise metric Lipschitz curve $\tilde{g}\in \mathcal{P}_{N_{\tilde{G}%
}}\left( \left[ 0,1\right] \text{,}\left( \mathcal{A}_{N}\left( Y\right)
,J_{\alpha }\right) \right) $ such that
\begin{equation*}
L_{J_{\alpha }}\left( \tilde{g}\right) =\mathbf{M}_{\alpha }\left( \tilde{G}%
\right) .
\end{equation*}%
Now, let
\begin{equation*}
g\left( t\right) =\left\{
\begin{array}{cc}
\tilde{g}\left( \frac{t}{\lambda }\right) , & 0\leq t\leq \lambda \\
\Gamma _{e}\left( \frac{t-\lambda }{1-\lambda }\right) , & \lambda \leq
t\leq 1%
\end{array}%
\right. ,
\end{equation*}%
where $\lambda =\frac{N_{G}-1}{N_{G}}$, and $\Gamma _{e}$ is the associated
geodesic in $Y$ from $e^{-}$ to $e^{+}$. Then, $g\in \mathcal{P}%
_{N_{G}}\left( \left[ 0,1\right] \text{,}\left( \mathcal{A}_{N}\left(
Y\right) ,J_{\alpha }\right) \right) $ and
\begin{equation*}
L_{J_{\alpha }}\left( g\right) =L_{J_{\alpha }}\left( \tilde{g}\right)
+L_{J_{\alpha }}\left( \Gamma _{e}\right) =\mathbf{M}_{\alpha }\left( \tilde{%
G}\right) +w\left( e\right) ^{\alpha }length\left( e\right) =\mathbf{M}%
_{\alpha }\left( G\right) .
\end{equation*}
\end{proof}

\begin{remark}
From this lemma, we see that for any transport path $G\in Path\left( \mathbf{%
a,b}\right) $ in a geodesic metric space $\left( Y,d\right) $, we have a
simple formula for the transport cost:
\begin{equation*}
\mathbf{M}_{\alpha }\left( G\right) =\int_{0}^{1}\left| \dot{g}\left(
t\right) \right| _{J_{\alpha }}dt.
\end{equation*}%
On the other hand, in \cite{buttazzo}, the authors studied another kind of
ramified transportation in which the cost of a path is given by
\begin{equation*}
\int_{0}^{1}\left| \dot{g}\left( t\right) \right| _{W}J\left( g\left(
t\right) \right) dt
\end{equation*}%
where $W$ is the Wasserstein distance on probability measures, and $J$ is
some function on the space of atomic probability measures. It is interesting
to see this difference between these two different approaches.
\end{remark}

\begin{theorem}
Suppose $\left( Y,d\right) $ is a geodesic metric space. Then,
$J_{\alpha }$ is a perfect quasimetric on $\mathcal{A}_{N}\left(
Y\right) $, and thus it induces a metric $D_{J_{\alpha }}$ on
$\mathcal{A}_{N}\left( Y\right) $.
\end{theorem}

\begin{proof}
Suppose $\mathbf{a,b}$ are two points in $\mathcal{A}_{N}\left( Y\right) $.
For any $f\in \mathcal{P}_{k}\left( \left[ 0,1\right] \text{,}\left(
\mathcal{A}_{N}\left( Y\right) ,J_{\alpha }\right) \right) $ with $f\left(
0\right) =\mathbf{a}$ and $f\left( 1\right) =\mathbf{b}$, there exists a
partition $P=\left\{ 0=a_{0}<\cdots <a_{k}=1\right\} $ of $\left[ 0,1\right]
$ such that $J_{\alpha }$ is a metric on $f\left( \left[ a_{i},a_{i+1}\right]
\right) $ and $f\lfloor _{\left[ a_{i},a_{i+1}\right] }$ is Lipschitz for
each $i=0,1,\cdots ,k-1$. Let $x_{i}=f\left( a_{i}\right) $ for each $i$, by
proposition \ref{M_alpha_J_alpha}, there exists a transport path $G$ from $%
f\left( 0\right) =\mathbf{a}$ to $f\left( 1\right) =\mathbf{b}$ such that
\begin{equation*}
\mathbf{M}_{\alpha }\left( G\right) \leq \sum J_{\alpha }\left(
x_{i},x_{i+1}\right) \leq \sum_{i}L\left( f\lfloor _{\left[ a_{i},a_{i+1}%
\right] }\right) =L\left( f\right)
\end{equation*}%
and $G$ contains no cycles.When $\left( Y,d\right) $ is a geodesic metric
space, each edge of $G$ is realized by a geodesic curve between its
endpoints. By lemma \ref{path2curve}, $G$\ determines a curve $g\in \mathcal{%
P}_{N_{G}}\left( \left[ 0,1\right] \text{,}\left( \mathcal{A}_{N}\left(
Y\right) ,J_{\alpha }\right) \right) $ with $L\left( g\right) =\mathbf{M}%
_{\alpha }\left( G\right) \leq L\left( f\right) $. Since $\mathbf{a},\mathbf{%
b}\in \mathcal{A}_{N}\left( Y\right) $ and $G\in Path\left( \mathbf{a,b}%
\right) $, the total number of vertices of $G$ with degree one is no more
than $2N$. Since $G$ contains no cycles, the total number $N_{G}$ of edges
of $G$ is no more than $4N-3$. Thus, $g\in \mathcal{P}_{4N-3}\left( \left[
0,1\right] \text{,}\left( \mathcal{A}_{N}\left( Y\right) ,J_{\alpha }\right)
\right) $. Hence, for any $\mathbf{a},\mathbf{b}\in \mathcal{A}_{N}\left(
Y\right) $,
\begin{equation*}
D_{J_{\alpha }}^{\left( k\right) }\left( \mathbf{a},\mathbf{b}\right)
=D_{J_{\alpha }}^{\left( 4N-3\right) }\left( \mathbf{a},\mathbf{b}\right)
\end{equation*}%
for any $k\geq 4N-3$. This shows that $J_{\alpha }$ is a perfect
quasimetric on $\mathcal{A}_{N}\left( Y\right) $.
\end{proof}

\begin{corollary}
\label{D_J_optimal}Suppose $\left( Y,d\right) $ is a geodesic metric space.
Then, for any $\mathbf{a,b\in }\mathcal{A}_{N}\left( Y\right) $ and $\alpha
\leq 1$, we have
\begin{equation*}
D_{J_{\alpha }}\left( \mathbf{a},\mathbf{b}\right) =\min \left\{ \mathbf{M}%
_{\alpha }\left( G\right) :G\in Path\left( \mathbf{a},\mathbf{b}\right)
\right\} .
\end{equation*}

\begin{proof}
Let $G$ be any optimal transport path from $\mathbf{a}$ to $\mathbf{b}$.
From the proof of the above theorem, we see $D_{J_{\alpha }}\left( \mathbf{a}%
,\mathbf{b}\right) \leq \mathbf{M}_{\alpha }\left( G\right) \leq L\left(
f\right) $ for any $f\in \mathcal{P}_{k}\left( \left[ 0,1\right] \text{,}%
\left( \mathcal{A}_{N}\left( Y\right) ,J_{\alpha }\right) \right) $ with $%
k\geq 4N-3$. Hence, $D_{J_{\alpha }}\left( \mathbf{a},\mathbf{b}\right) =%
\mathbf{M}_{\alpha }\left( G\right) $.
\end{proof}
\end{corollary}

\begin{corollary}
\label{length_space_proof}\bigskip Suppose $\left( Y,d\right) $ is a
geodesic metric space. Then, $(\mathcal{A}_{N}\left( Y\right) ,D_{J_{\alpha
}})$ is a length space.
\end{corollary}

\begin{proof}
\bigskip By corollary\ \ref{D_J_optimal}, each optimal transport path $G$
determines a solution $g$ to the geodesic problem (\ref{geodesic_problem}).
Then, by theorem \ref{length_space}, $(\mathcal{A}_{N}\left( Y\right)
,D_{J_{\alpha }})$ becomes a length space.
\end{proof}

Since $\mathcal{A}_{1}\left( Y\right) \subset \mathcal{A}_{2}\left( Y\right)
\subset \cdots \subset \mathcal{A}_{N}\left( Y\right) \subset \cdots $, and $%
(\mathcal{A}_{N}\left( Y\right) ,D_{J_{\alpha }})$ is a length space for
each $N$, we have

\begin{proposition}
\label{all_atomic_measures}Suppose $\left( Y,d\right) $ is a geodesic metric
space. Then, $D_{J_{\alpha }}$ is a metric on the space $\mathcal{A}\left(
Y\right) $ of all atomic probability measures on $Y$. Moreover, $\left(
\mathcal{A}\left( Y\right) ,D_{J_{\alpha }}\right) $ is a length space.
\end{proposition}

We now give some conclusive remarks:

\begin{remark}
In \cite{xia1}, we defined $d_{\alpha }\left( \mathbf{a},\mathbf{b}\right)
:=\min \left\{ \mathbf{M}_{\alpha }\left( G\right) :G\in Path\left( \mathbf{a%
},\mathbf{b}\right) \right\} $ for $0\leq \alpha <1$ and showed that $%
d_{\alpha }$ defines a metric on the space of (atomic) probability measures.
Moreover, we showed $\left( \mathcal{A}\left( Y\right) ,d_{\alpha }\right) $
is a length space. Now, from corollary \ref{D_J_optimal}, we see that $%
d_{\alpha }=D_{J_{\alpha }}$. That is, the metric $d_{\alpha }$ is just the
intrinsic metric on $\mathcal{A}\left( Y\right) $ induced by the quasimetric $%
J_{\alpha }$. Proposition \ref{all_atomic_measures} simply gives another
proof of $\left( \mathcal{A}\left( Y\right) ,d_{\alpha }\right) $ being a
length space. Furthermore, an optimal transport path studied in \cite{xia1}
is simply a geodesic in the length space $\left( \mathcal{A}\left( Y\right)
,D_{J_{\alpha }}\right) $.
\end{remark}

\begin{remark}
Suppose $\left( Y,d\right) $ is a geodesic metric space, and $\mathcal{P}%
_{\alpha }\left( Y\right) $ is the completion of the metric space $\left(
\mathcal{A}\left( Y\right) ,D_{J_{\alpha }}\right) $. Then, \ $\left(
\mathcal{P}_{\alpha }\left( Y\right) ,D_{J_{\alpha }}\right) $ is also a
length space. A geodesic in the length space $\left( \mathcal{P}_{\alpha
}\left( Y\right) ,D_{J_{\alpha }}\right) $ is also called an optimal
transport path between its endpoints.
\end{remark}

\end{document}